\documentclass[a4paper]{amsart}
\usepackage{amsfonts,amssymb,amsmath,amsthm}
\usepackage{url}
\usepackage{enumerate}

\newtheorem{Lemma}{Lemma}[section]\newcommand{\bel}{\begin{Lemma}}\newcommand{\eel}{\end{Lemma}}
\newtheorem{Proposition}[Lemma]{Proposition}\newcommand{\bprop}{\begin{Proposition}}\newcommand{\eprop}{\end{Proposition}}
\newtheorem{Theorem}[Lemma]{Theorem}\newcommand{\bthe}{\begin{Theorem}}\newcommand{\ethe}{\end{Theorem}}
\newtheorem{Remark}[Lemma]{Remark}\newcommand{\beR}{\begin{Remark}\rm}\newcommand{\eeR}{\end{Remark}}
\newtheorem{Definition}[Lemma]{Definition}\newcommand{\bed}{\begin{Definition}}\newcommand{\eed}{\end{Definition}}
\newtheorem{Example}[Lemma]{Example}\newcommand{\bex}{\begin{Example}\rm}\newcommand{\eex}{\end{Example}}
\newtheorem{Corollary}[Lemma]{Corollary}\newcommand{\bcor}{\begin{Corollary}}\newcommand{\ecor}{\end{Corollary}}

\def\bull{\vrule height .9ex width .9ex depth -.1ex }
\def\bpr{{\em Proof:~}}\def\epr{$\bull$\\}

\def\isom{\xrightarrow{\sim}}

\def\la{\lambda}\def\si{{\sigma}}

\def\cN{\mathcal{N}}

\def\ud{{\underline{d}}}

\newcommand{\beq}{\begin{equation}}\newcommand{\eeq}{\end{equation}}
\newcommand{\ber}{\begin{array}{lll}}\newcommand{\eer}{\end{array}}
\newcommand{\bpm}{\begin{pmatrix}}\newcommand{\epm}{\end{pmatrix}}

\theoremstyle{definition}

\newcommand{\g}{\ensuremath{\mathfrak g}}

\newcommand{\m}{\ensuremath{\mathsf M}}

\newcommand{\Z}{\ensuremath{{\mathbb Z}}}

\newcommand{\cO}{\ensuremath{\mathcal O}}

\newcommand{\R}{\ensuremath{\mathbb R}}
\newcommand{\C}{\ensuremath{\mathbb C}}
\def\P{\ensuremath{\mathbb P}}

\renewcommand{\Pr}{\ensuremath{\mathbb P}}

\def\suml{\sum\limits}
\def\capl{\mathop\cap\limits}

\def\prodl{\prod\limits}
\def\smin{\setminus}\def\sset{\subset}

\def\li{~\\ $\bullet$ }

\author{Tatyana Barron}
\address{T. Barron, Department of Mathematics,
University of Western Ontario, London, Ontario N6A 5B7, Canada }
\email{tatyana.barron@uwo.ca}

\author{Dmitry Kerner}
\address{D. Kerner, Department of Mathematics, Ben-Gurion University of the Negev, P.O.B. 653, Be'er Sheva 84105, Israel}
\email{dmitry.kerner@gmail.com}

\author{Marina Tvalavadze}
\address{M. Tvalavadze, Fields Institute for Research in Mathematical Sciences,
Toronto, Ontario, M5T 3J1, Canada}
\email{mtvalava@fields.utoronto.ca}

\thanks{2010 Mathematics Subject Classification: Primary 17B70, Secondary 14F45}

\thanks{Research of the first author is supported in part
by NSERC}

\thanks{D.K. was partially supported by the grant FP7-People-MCA-CIG, 334347}

\thanks{The first author previously published papers under
the name Tatyana Foth}

\keywords{Filiform Lie algebras, graded Lie algebras, projective varieties, topology, classification}

\title{On varieties of Lie algebras of maximal class}

\begin{document}
\sloppy

\maketitle
\begin{abstract}
We study complex projective varieties that parametrize
(finite-dimensional) filiform Lie algebras over ${\mathbb C}$,
using equations derived by Millionshchikov. In the
infinite-dimensional case we concentrate our attention on
${\mathbb N}$-graded Lie algebras of maximal class. As shown by A.
Fialowski (see also \cite{shalev:97, mill:04}) there are only
three isomorphism types of $\mathbb{N}$-graded Lie algebras
$L=\oplus^{\infty}_{i=1} L_i$ of maximal class generated by $L_1$
and $L_2$, $L=\langle L_1, L_2 \rangle$.  Vergne described the
structure of these algebras with the property $L=\langle L_1
\rangle$. In this paper we study those generated by the first and
$q$-th components where $q>2$, $L=\langle L_1, L_q \rangle$. Under
some technical condition, there can only be one isomorphism type
of such algebras. For $q=3$ we fully classify them.  This gives a
partial answer to a question posed by Millionshchikov.
\end{abstract}

\section{Introduction}

There has been a lot of efforts to understand the algebraic
varieties that parametrize Lie algebras, in particular nilpotent
Lie algebras. The long list of literature on this subject
includes, in particular, \cite{ver:70}, \cite{kir:87},
\cite{hakim:91}, \cite{echar:08}.

A {\it filiform} Lie algebra is a nilpotent Lie algebra of maximal class
of nilpotency. A generalization of this concept, called {\it a Lie
algebra of maximal class}, and varieties of such algebras were
studied in \cite{mill:09}.

A Lie algebra $\g$ is called {\it residually nilpotent} if
$\cap_{i=1}^{\infty}\mathrm{g}^i=\{ 0\}$, where $\mathrm{g}^1
=\mathrm{g}$ and $\{ \mathrm{g}^i \}$ is the lower central series
of $\mathrm{g}$. A residually nilpotent Lie algebra $\mathrm{g}$
is called {\it a Lie algebra of maximal class} if $\Sigma
_i(\dim\, \mathrm{g}^i/\mathrm{g}^{i+1} -1)=1$. In the
finite-dimensional case these are exactly filiform Lie algebras.

An explicit system of quadratic equations that describes a variety
of filiform Lie algebras or of Lie algebras of maximal class is
provided in \cite{mill:09}, - this is one of the main results of
the paper. We use this system to study topology of the varieties
of $n$-dimensional filiform complex Lie algebras (Section
\ref{topology}).

In Section \ref{algebra} and in Appendix \ref{app2} we study
central extensions of  $\mathbb{N}$-graded filiform Lie algebras
with lacunas in grading. Although neither finite-dimensional
filiform Lie algebras nor infinite-dimensional Lie algebras of
maximal class have been classified up to an isomorphism they have
been extensively studied for the last few decades. In
\cite{mill:09}  Millionshchikov conjectured the existence of only
three isomorphism types of Lie algebras of maximal class generated
by the first and the $q$-th graded components (i.e. with lacunas
in grading from 2 to $q-1$). Using the  results on central
extensions from subsections 3.1 and 3.2 we prove Theorem
\ref{MainTheorem} and also classify $\mathbb{N}$-graded Lie
algebras of maximal class generated by the first and third graded
components. Our results coincide with computer calculations of
Vaughan-Lee who investigated the cases $q=3,4$.

Our main results are Theorems \ref{The.Case.n=10},
\ref{The.Case.n=11}, \ref{MainTheorem} and \ref{caseq3}.

{\bf Acknowledgements.} We are thankful to M. Vergne for telling
us about reference \cite{mill:09}. We thank A. Dhillon and M. Khalkhali for
brief but helpful discussions. We are grateful to the referee for suggestions.

\section{Topology of parameter spaces}\label{topology}

In this section all Lie algebras are finite-dimensional, over
$\C$.

Equations of the affine variety of $n$-dimensional
filiform complex Lie algebras are given in
\cite{mill:09} (equations (28), (29), (19), (30)).
We are going to consider, instead, complex
projective varieties $\m _n$ that are obtained from the varieties
in Theorem 4.3 \cite{mill:09} by removing the Abelian Lie algebra,
projectivizing, and "dropping" the variables that do not appear in
the equation - see a precise explanation below.

\subsection{General properties of $\m _n$}
\label{Mns}
The parameter space for $n\ge 9$ is defined by the following system of
quadratic equations on the variables $x_{j,s}$,
$s=0,...,n-5$, $j=2,...,\frac{[n-1]}{2}$,
(Theorem 4.3 \cite{mill:09}):
\li for n-odd: $\m _n=\Big\{F_{j,q,r}=0| \ 2\le
j< q,\ j+2q+r+1\le n,\quad
r\ge0\Big\}$
\li for n-even: $\m _n=\Bigg\{\ber F_{j,q,r}=0, \ 2\le j< q,\ j+2q+r+1<n,\   r\ge0\\
F_{j,q,r}+(-1)^{\frac{n}{2}-j-q}x_{-1}G_{j,q,r}=0,\ 2\le j< q,\ j+2q+1+r=n,\ r\ge0 \\
x_{-1}G_{j,q,-1}=0,\ 2\le j< q, \ j+2q=n\eer\Bigg\}$
where polynomials $F_{j,q,r}$ and $G_{j,q,r}$ are as follows:
$$
F_{j,q,r}=\sum_{t=0}^r \sum_{l=j}^{[\frac{j+q-1}{2}]} \sum_{m=q+1}^{q+[\frac{j+t}{2}]}
(-1)^{l-j+m-q}\begin{pmatrix}q-l-1 \cr l-j\end{pmatrix}
\begin{pmatrix}j+q-m+t-1 \cr m-q-1\end{pmatrix}
x_{l,t}x_{m,r-t}+
$$
$$
\sum_{t=0}^r \sum_{l=j}^{[\frac{j+q}{2}]} \sum_{m=q}^{q+[\frac{j+t}{2}]}
(-1)^{l-j+m-q}\begin{pmatrix}q-l \cr l-j\end{pmatrix}
\begin{pmatrix}j+q-m+t \cr m-q\end{pmatrix}
x_{l,t}x_{m,r-t}+
$$
$$
\sum_{t=0}^r \sum_{m=j}^{q+[\frac{j+t}{2}]}
(-1)^{m-j+1}\begin{pmatrix}2q-m+t \cr m-j\end{pmatrix}
x_{q,t}x_{m,r-t},
$$
$$
G_{j,q,r}=\sum_{l=j}^{[\frac{j+q-1}{2}]} (-1)^l\begin{pmatrix}q-l-1 \cr l-j\end{pmatrix}
x_{l,r+1}+
\sum_{l=j}^{[\frac{j+q}{2}]} (-1)^l\begin{pmatrix}q-l \cr l-j\end{pmatrix}
x_{l,r+1}
-(-1)^q x_{q,r+1}.
$$
The variable $x_{-1}$ is denoted by $x$ in \cite{mill:09}. Our
notation emphasizes the weight. The quadratic polynomials
$F_{j,q,r}$ are weighted homogeneous, they contain only monomials
of the form $x_{a_1,t}x_{a_2,r-t}$. The linear forms $G_{j,q,r}$
contain only monomials of the  form $x_{a,r+1}$. Thus, for $n$-odd
or for $n$-even and $x_{-1}=0$ the variables $x_{a,t}$ with
$t>n-9$ do not participate in the defining equations. Moreover,
even for $t\le n-9$ many of $x_{a,t}$ do not appear in the
defining equations. We shall always consider the "minimal" version
of $\m _n$, inside the projective space spanned by those $x_{a,t}$
that do appear. Adding more variables that do not participate in
the equations means taking the cone over the "minimal" $\m _n$.

\

\

The group $\C^*$ acts on $\m _n$ by weighted homogeneous scaling
\beq
\Big(x_{-1},\{x_{*,0}\},\{x_{*,1}\},\dots,\{x_{*,n-5}\}\Big)\to
\Big(\la^{-1}x_{-1},\{x_{*,0}\},\la\{x_{*,1}\},\dots,\la^{n-5}\{x_{*,n-5}\}\Big),\quad
\la\in\C^* \eeq This is seen directly, the defining equations are
equivariant. The points of each $\C^*$ orbit correspond to
isomorphic Lie algebras, as the action is induced by the rescaling
of the generators, $e_i\to \la^{i-1}e_i$.

Accordingly, the set of homogeneous coordinates splits into the
subsets $\{x_{*,s}\}_s$, coordinates of the same $\C^*$-weights.
The loci $\m_n\cap\{x_{*,i}=0,\ for\ i\neq j\}$ are invariant under
this $\C^*$ action. The parameter space gets decomposed into
'elementary pieces' as described in the Proposition below.
\bprop Let $n\ge 9$. \\
1. For $n$-odd, the open
dense part $\m_n\smin\{x_{*,0}=0\}$ deforms (homotopically) to
$\m_n\cap \{x_{*,i}=0,\ for\ i>0\}$. For $n$-even,
$\m_n\smin\{x_{-1}=0\}$ deforms (homotopically) to $\m _n\cap
\{x_{*,i}=0,\ for\ i\ge 0\}$, which is a point. In both cases the
open dense part $\m_n\smin\{x_{*,n-5}=0\}$ deforms (homotopically)
to $\m_n\cap \{x_{*,i}=0,\ for\ i<n-5\}$.
\\2. The topological Euler characteristic of $\m_n$ can be computed as
$\chi(\m_n)=\suml_{j=-1}^{n-5}\chi\Big(\m_n\cap\{x_{*,i}=0,\ for\
i\neq j\}\Big)$.
\\3. For $n$-odd $\m_n\cap\{x_{*,i}=0,\ for \ 0\le i\le\frac{n-9}{2}\}$ is isomorphic to the projective space
with homogeneous coordinates $\{x_{j,k}\}_{2\le
j\le\frac{n-1}{2}}^{\frac{n-9}{2}+1\le k\le n-5}$.
\\4.
$\m_n\cap\{x_{*,i}=0, \ for \ i\neq[\frac{n-9}{2}]\}=\{F_{2,3,2[\frac{n-9}{2}]}=0\}\sset\P(\{x_{j,[\frac{n-9}{2}]}\}_{j=2\dots[\frac{n-1}{2}]})$.
\\5. For $n>10$: $\m_n\cap\{x_{*,i}=0, \ for \ i\neq[\frac{n-9}{2}]-1\}=\{F_{2,3,n-11}=0=F_{2,4,n-11}\}$ for $n$-odd and $\{F_{2,3,n-12}=0=F_{2,4,n-12}=F_{3,4,n-12}\}$
for $n$-even. \eprop \bpr {\bf 1.} The case of odd $n$. Suppose at
least one of $x_{*,0}$ is not zero.
 Consider the flow $(x_{*,0},\la x_{*,1},\dots,\la^{n-5} x_{*,n-5})$, for $\la\in\C^*$.
This is a continuous  flow, for $\la\in\C$, well-defined on
$\m_n\smin\{x_{*,0}=0\}$.
 It provides the homotopy from $\m_n\smin\{x_{*,0}=0\}$, for $\la=1$, to $\m_n\cap \{x_{*,i}=0,\ for\ i>0\}$, for $\la=0$.
 The other statements in 1. are proved similarly.

{\bf 2.} Suppose $n$ is odd (the proof for $n$-even is similar).
We shall use 1. and similar statements:

\noindent $(\m_n\cap \{ x_{*,0}=0\} )-\{ x_{*,1}=0\}$ deforms homotopically
to $\m_n\cap \{ x_{*,i}=0 \ for \ i\neq 1\}$,

\noindent $(\m_n\cap \{ x_{*,0}=0, x_{*,1}=0\} )-\{ x_{*,2}=0\}$ deforms homotopically
to $\m_n\cap \{ x_{*,i}=0 \ for \ i\neq 2\}$, etc.

By the additivity of Euler characteristic
\beq\ber
\chi(\m_n)=\chi(\m_n\cap \{ x_{*,0}=0\} )+\chi(\m_n\backslash \{x_{*,0}=0\} )
=
\\
\chi(\m_n\cap \{ x_{*,0}=0=x_{*,1}\} )+
\chi((\m_n\cap \{ x_{*,0}=0\} ) \backslash \{ x_{*,1}=0\} )+
\chi(\m_n\cap\{ x_{*,i}=0 \ for \
i>0 \} )=\dots\\\dots= \suml_{j=0}^{n-5}\chi(\m_n\cap\{x_{*,i}=0,\
for\ i\neq j\})
\eer\eeq {\bf 3.} The polynomial $F_{j,q,r}$
consists of monomials $x_{*,t}x_{*,r-t}$, hence if $\{x_{*,i}=0,\
for \ 0\le i\le[\frac{n-9}{2}]\}$ and $r\le n-9$ then all the
relevant polynomials vanish. So, $\m_n\cap\{x_{*,i}=0,\ for \ 0\le
i\le[\frac{n-9}{2}]\}$ is the projective space, spanned by the
remaining coordinates.

{\bf 4. 5.} By direct check of the defining equations. \epr

\subsection{The parameter space for n=9,10,11}

In this subsection we quickly discuss $\m_9$, after that
compute Betti numbers of the components of $\m_{n}$ for $n=10,11$
and discuss geometric structure of these components.
Some facts that we use are summarized in Appendix \ref{app1}.

\subsubsection{n=9} The subvariety $\m_9\sset\Pr ^2$ is defined by the equation
\begin{equation}\label{F230}
2x_{2,0}x_{4,0}-3x_{3,0}^2+x_{3,0}x_{4,0}=0,
\end{equation}
where $x_{2,0}$, $x_{3,0}$, $x_{4,0}$ are the coefficients that
appear in the deformation cocycle. Remark: (\ref{F230}) is
consistent with the first equation in (22) \cite{mill:09}, in the
notations of \cite{mill:09} this is the equation $F_{2,3,0}=0$.

This subvariety (an algebraic curve of degree 2) is smooth. (Its
singular locus is $x_{2,0}=0=x_{4,0}=x_{3,0}$ i.e. it is empty.) By
the genus formula the genus is $(2-1)(2-2)/2=0$, thus $\m_9$ is
isomorphic to $\P^1$.

\subsubsection{n=10}
The subvariety $\m_{10}\sset\P(x_{-1},\{x_{j,s}\})$ is defined by
\beq \left\{ \begin{array}{lll} F_{2,3,0}: &
2x_{2,0}x_{4,0}-3x_{3,0}^2+x_{3,0}x_{4,0}=0
\\F_{2,3,1}+xG_{2,3,1}: &-2x_{2,0}x_{4,1}+7x_{3,0}x_{3,1}-x_{3,0}x_{4,1}-3x_{4,0}x_{2,1}-3x_{4,0}x_{3,1}+x_{-1}(2x_{2,2}+x_{3,2})=0
\\x_{-1}G_{2,4,-1}: & x_{-1}(2x_{2,0}-x_{3,0}-x_{4,0})=0 \end{array} \right.
\eeq where $x_{i,j}$ are the coefficients that appear in the
deformation cocycle.

Denote by $\m _{10}^{(0)}$ the component of $\m_{10}$ that
corresponds to $x_{-1}=0$. Its equations are
\begin{equation}
\label{firstcomp} \left\{ \begin{array}{ll}
2x_{2,0}x_{4,0}-3x_{3,0}^2+x_{3,0}x_{4,0}=0 \\
-2x_{2,0}x_{4,1}+7x_{3,0}x_{3,1}-x_{3,0}x_{4,1}-3x_{4,0}x_{2,1}
-3x_{4,0}x_{3,1}=0.  \end{array} \right.
\end{equation}
Note: the first equation is the same as (\ref{F230}), and the
second equation ($F_{2,3,1}=0$ in notations of \cite{mill:09})
coincides with the third equation of Example 4.6 \cite{mill:09}.

To simplify the formulas we make a linear change of coordinates:
$$\begin{array}{ccc}
z_{0,0}=2x_{2,0}+x_{3,0}   & z_{1,0}=x_{3,0} & z_{2,0}=x_{4,0} \\
z_{0,1}=3(x_{2,1}+x_{3,1}) & z_{1,1}=x_{3,1} & z_{2,1}=x_{4,1}
\end{array}$$
In the new coordinates $\m _{10}^{(0)}$ is defined as
$\{x_{-1}=z_{0,0}z_{2,0}-3z_{1,0}^2=0=-z_{0,0}z_{2,1}+7z_{1,0}z_{1,1}-z_{2,0}z_{0,1}\}\subset\Pr^6_{x_{-1},z_{*,*}}$.

Denote by $\m_{10}^{(1)}$ the component of
$\m_{10}$ that corresponds to $x_{-1}\neq0$. For it we get the
system
\begin{equation}
\label{secondcomp} \left\{ \begin{array}{ll}
2x_{2,0}x_{4,0}-3x_{3,0}^2+x_{3,0}x_{4,0}=0 \\
-2x_{2,0}x_{4,1}+7x_{3,0}x_{3,1}-x_{3,0}x_{4,1}-3x_{4,0}x_{2,1}
-3x_{4,0}x_{3,1}+x_{-1}(2x_{2,2}+x_{3,2})=0 \\
2x_{2,0}-x_{3,0}-x_{4,0}=0, \end{array} \right.
\end{equation}
As before, change the coordinates:
$$\begin{array}{cccc}
z_{0,0}=2x_{2,0}+x_{3,0}   & z_{1,0}=x_{3,0} & z_{2,0}=x_{4,0}& z_{0,1}=3(x_{2,1}+x_{3,1})\\
z_{1,1}=x_{3,1} & z_{2,1}=x_{4,1}& z_{0,2}=x_{3,2}&
z_{1,2}=2x_{2,2}+x_{3,2}
\end{array}$$
Then the equations become:
$\{z_{0,0}z_{2,0}-3z_{1,0}^2=0=-z_{0,0}z_{2,1}+7z_{1,0}z_{1,1}-z_{2,0}z_{0,1}+x_{-1}z_{1,2}=z_{0,0}-2z_{1,0}-z_{2,0}\}
\subset\Pr^7_{x_{-1},z_{*,*}}$. Eliminate $z_{0,0}$, using the
linear equation to get the equivalent presentation:
\beq\label{Eq.Presentation.of.M.1.10}
\m_{10}^{(1)}=\{(3z_{1,0}+z_{2,0})(z_{1,0}-z_{2,0})=0,\
(2z_{1,0}+z_{2,0})z_{2,1}=7z_{1,0}z_{1,1}-z_{2,0}z_{0,1}+x_{-1}z_{1,2}\}\sset\P^6
\eeq

\bthe
\label{The.Case.n=10}
1. $\m _{10}^{(0)}\subset \Pr ^6$ is a subvariety of
dimension 3, its singular locus is the plane
$\{x_{-1}=z_{0,0}=0=z_{1,0}=z_{2,0}\}\subset \Pr ^6$. The
singularities of $\m _{10}^{(0)}$ are resolved in one blowup and
the resolution, $\widetilde{\m _{10}^{(0)}}$, is a projective
bundle over $\Pr^1$. More explicitly: $\widetilde{\m
_{10}^{(0)}}=\Pr\Big(T^*_{\Pr^2}\oplus\cO_{\Pr^2}(-1)\Big)|_{C_2}$,
where $T^*_{\Pr^2}$ is the cotangent bundle of the plane and
$|_{C_2}$ denotes the restriction of the bundle onto a smooth
conic $C_2\subset\Pr^2$.

$\m _{10}^{(0)}$ admits an algebraic cell structure, $\m
_{10}^{(0)}=\C^3\cup(2\C^2)\cup\C^1\cup\C^0$. In particular, its
odd homologies vanish, while the even homologies are: $H_0(\m
_{10}^{(0)},\Z)=\Z=H_6(\m _{10}^{(0)},\Z)$, $H_2(\m
_{10}^{(0)},\Z)=\Z$, $H_4(\m _{10}^{(0)},\Z)=\Z^{\oplus2}$.

2. $\m _{10}^{(1)}$ is isomorphic to the union of two quadric
3-folds in $\Pr^5$, intersecting along
$\{z_{0,0}=0=z_{1,0}=z_{2,0}=x_{-1}z_{1,2}\}$. Each
quadric is the cone over a smooth quadric surface in $\P^3$. \ethe
\noindent\bpr {\bf 1.} We consider everything inside the
hyperplane $\{x_{-1}=0\}=\P^5\sset\P^6$. The singular locus of
$\m_{10}^{(0)}$ is defined by the maximal minors of Jacobian
matrix (the matrix of all the partials of the defining equations):
\beq\bpm
z_{2,0}&-6z_{1,0}&z_{0,0}&0&0&0\\
-z_{2,1}&7z_{1,1}&-z_{0,1}&-z_{2,0}&7z_{1,0}&-z_{0,0} \epm\eeq
Thus this singular locus is $Z:=\{z_{0,0}=0=z_{1,0}=z_{2,0}\}.$
(We consider the singular locus as a set, omitting the
multplicities.) Note that the singular locus is of codimension one
in $\m_{10}^{(0)}$, i.e. the variety is not normal. To compute the
normalization, i.e. the normal variety $\widetilde{\m_{10}^{(0)}}$
with the finite surjective birational morphism,
$\widetilde{\m_{10}^{(0)}}\to\m_{10}^{(0)}$, we blowup along the
singular locus: \beq
\widetilde{\m_{10}^{(0)}}:=Bl_Z\m_{10}^{(0)}=\Big\{(z_{0,0},z_{1,0},z_{2,0})\sim(\si_0,\si_1,\si_2),\quad
\si_0\si_2=3\si^2_1, \quad -\si_0z_{2,1}+7\si_1 z_{1,1}-\si_2
z_{0,1}=0\Big\}\sset\P^5_{z_{*,*}}\times\P^2_{\si_*} \eeq By
construction $\widetilde{\m_{10}^{(0)}}$ has two natural
projections:
$\widetilde{\m_{10}^{(0)}}\stackrel{\pi_\si}{\to}\P^2_{\si_*}$
and \
 $\widetilde{\m_{10}^{(0)}}\stackrel{\pi_z}{\to}\m_{10}^{(0)}$.

Consider the fibres of $\pi_z$. If
$(z_{0,0},z_{1,0},z_{2,0})\neq(0,0,0)$ then the condition
$(z_{0,0},z_{1,0},z_{2,0})\sim(\si_0,\si_1,\si_2)$ determines the
point $(\si_0,\si_1,\si_2)\in\P^2_{\si_*}$ uniquely. If
$(z_{0,0},z_{1,0},z_{2,0})=(0,0,0)$ then
$(z_{0,1},z_{1,1},z_{2,1})\neq(0,0,0)$ and therefore there are two
equations on $(\si_0,\si_1,\si_2)$: linear and quadratic.
Geometrically, for a fixed triple $(z_{0,1},z_{1,1},z_{2,1})$ we
have a line and a conic in the plane $\P^2_{\si_*}$, they
intersect in two points (counted with multiplicity). Therefore the
projection  $\pi_z$ is {\em finite} and is a 2:1 cover over the
singular locus.

Consider the projection $\pi_\si$. For a fixed point
$(\si_0,\si_1,\si_2)\in\P^2_{\si_*}$ the conditions on $z_{*,*}$
are {\em linear} and linearly independent. Therefore this
projection equips $\widetilde{\m_{10}^{(0)}}$ with the structure
of $\P^2$ bundle over its image, the conic
$\{\si_0\si_2=3\si^2_1\}\sset\P^2_{\si_*}$. In particular, it
follows that $\widetilde{\m_{10}^{(0)}}$ is smooth, hence the map
$\widetilde{\m_{10}^{(0)}}\stackrel{\pi_z}{\to}\m_{10}^{(0)}$ is
not only normalization but also a resolution
 of singularities. Finally note that $\pi_\si(\widetilde{\m_{10}^{(0)}})\sset\P^2$ is a smooth conic $C\sset\P^2_\si$, therefore $\widetilde{\m_{10}^{(0)}}$ is a $\P^2$ bundle over $\P^1$.

It remains to understand the $\P^2$ bundle structure. First,
consider the particular locus in $\widetilde{\m_{10}^{(0)}}$,
where $z_{0,0}=z_{1,0}=z_{2,0}=0$.  We claim that
$\widetilde{\m_{10}^{(0)}}|_{z_{0,0}=z_{1,0}=z_{2,0}=0}=\P
T^*_{\P^2}|_C$.

Indeed, after rescaling the coordinates
$(z_{2,1},z_{1,1},z_{0,1})$, we get the defining equation
$\Big\{(\si_0,\si_1,\si_2)\times(z_{2,1},z_{1,1},z_{0,1})=0\Big\}\sset\P^2_z\times\P^2_\si$.
Geometrically this can be interpreted as the variety of pairs: a
point in $\P^2_\si$ and the lines passing through this point.
(Recall that a line in $\P^2$ is defined by a 1-form.) The set of
line passing through a point, i.e. the set of non-zero 1-forms up
to scaling, is naturally the projectivization of the cotangent
space to $\P^2_\si$ at this point. Therefore
$\widetilde{\m_{10}^{(0)}}|_{z_{0,0}=z_{1,0}=z_{2,0}=0}=\P
T^*_{\P^2}|_C$.

As the other extremal case, consider the locus
$z_{2,1}=z_{1,1}=z_{0,1}=0$.  Then
$(z_{0,0},z_{1,0},z_{2,0})\neq(0,0,0)$ and the fibre over
$(\si_0,\si_1,\si_2)$ is one point: the projectivization of the
tautological bundle, $\cO_{\P^2}(-1)$.

Finally, for any fixed point $\si\in\P^2_\si$ the remaining
equations are linear.  Hence each fibre over $\si$ is the span of
$\P T^*_{\P^2,\si}$ and $\P\cO_{\P^2}(-1)_\si$, i.e. each fibre is
$\P(T^*_{\P^2,\si}\oplus\cO_{\P^2}(-1)_\si)$. As this
identification is canonical we get the statement.

{\bf 2.} Follows from the presentation in equation
(\ref{Eq.Presentation.of.M.1.10}). \epr

\beR \noindent {\bf Singular points of  $\m_{10}^{(0)}$:} as
mentioned above, the singular points of $\m_{10}^{(0)}$ are
\[[z_0: z_1:z_2:z_3:z_4:z_5]\] such that $z_0=z_1=z_2=0$ and $z_3,
z_4$ and $z_5$ are any (but not all zero at the same time). This
means that $x_{2,0} = x_{3,0}= x_{4,0} = 0$ and $x_{2,1},$
$x_{3,1},$ $x_{4,1}$ take any values (not all zeros). Therefore,
the deformation cocycle $\Psi$ of $\mathrm{m}_0(10)$ (see
definition of $\mathrm{m}_0(n)$ in Section 3) is of the form
\begin{equation} \Psi = x_{2,1} \Psi_{2,1} + x_{3,1} \Psi_{3,1} +
x_{4,1} \Psi_{4,1}
\end{equation}
where $\Psi_{2,1},$ $\Psi_{3,1}$ and $\Psi_{4,1}$ are closed
cocycles from $H^2_{+}(\mathrm{m}_0(10),\mathrm{m}_0(10))$ defined
in on page 183 of \cite{mill:09} by explicit formulas. Then, for
$i,j>1$ $\Psi(e_i, e_j)= \alpha_{ij}e_{i+j+1}$ for appropriate
scalars $\alpha_{ij}$, and the cocycle vanishes on the remaining
vector pairs. Let us re-name basis elements as follows
\[ f_1=e_1,\quad f_{i+1} = e_i,\quad i=2,\ldots,10.\]
This is an $\mathbb{N}$-graded basis for a filiform Lie algebra
$\mathrm{g}$ corresponding to $\Psi$. Indeed,
\[ [f_i,f_j] = [e_{i-1}, e_{j-1}] = \alpha_{i-1,j-1} e_{i+j-1} =
\alpha_{i-1,j-1} f_{i+j}\] where $i,j >1$ and
\[ [f_1, f_i]= [e_1, e_{i-1}] = e_i = f_{i+1}\] where $i>1$. As shown in
Section 3  $\mathrm{g}$ must be one of the following algebras:
$\mathrm{m}^3_{0,1}(11)$ (the central extension of
$\mathrm{m}^3_0(10)$), $\mathrm{m}^3_{0,3}(11)$ (the third central
extension of $\mathrm{m}^3_0(8)$), $\mathrm{m}^3_{0,5}(11)$ (the
fifth central extension of $\mathrm{m}^3_0(6)$).

\eeR

\subsection{n=11}
The equations of $\m _{11}$ are, \cite{mill:09}: \beq\ber
F_{2,3,0}: & 2x_{2,0}x_{4,0}-3x_{3,0}^2+x_{3,0}x_{4,0}=0,\\
F_{2,3,1}: & -2x_{2,0}x_{4,1}+7x_{3,0}x_{3,1}-x_{3,0}x_{4,1}-3x_{4,0}x_{2,1}-3x_{4,0}x_{3,1}=0,\\
F_{2,4,0}: & -2x_{2,0}x_{5,0}+4x_{3,0}x_{4,0}-6x_{4,0}^2+x_{3,0}x_{5,0}+x_{4,0}x_{5,0}=0,\\
F_{2,3,2}: &
-2x_{2,0}x_{4,2}+8x_{3,0}x_{3,2}-x_{3,0}x_{4,2}-4x_{4,0}x_{2,2}-6
x_{4,0}x_{3,2}+
 2x_{5,0}x_{2,2}+x_{5,0}x_{3,2}-3x_{2,1}x_{4,1}+4x_{3,1}^2-3x_{3,1}x_{4,1}=0.
\eer\eeq As it was pointed out earlier, our expressions for
$F_{2,3,0}$, and $F_{2,3,1}$ are consistent with \cite{mill:09}.
Our equation $F_{2,4,0}=0$ is the same as the corresponding
equation in the Example 4.6 \cite{mill:09}. Our equation
$F_{2,3,2}=0$ differs from that in Example 4.6 \cite{mill:09} by
the coefficient at the $x_{2,2}x_{4,0}$ term, but we checked our
calculation several times and we are confident that our
coefficient is correct.

Apply the following change of variables: \beq\begin{array}{ccccc}
z_{0,0}=2x_{2,0}+x_{3,0}   & z_{1,0}=x_{3,0} & z_{2,0}=x_{4,0}& z_{3,0}= x_{5,0}-6x_{4,0}\\
z_{0,1}=3(x_{2,1}+x_{3,1}) & z_{1,1}=x_{3,1} & z_{2,1}=x_{4,1}\\
z_{0,2}= 2x_{2,2}+x_{3,2} & z_{1,2}= x_{3,2} & z_{2,2}= x_{4,2}.
\end{array}
\eeq Then we get the following equations: \beq\ber F_1: &
z_{0,0}z_{2,0}-3z_{1,0}^2=0
\\F_2: & -z_{0,0}z_{2,1}+7z_{1,0}z_{1,1}-z_{2,0}z_{0,1}=0
\\F_3: & z_{3,0}(2z_{1,0}-z_{0,0}+z_{2,0})+z_{2,0}(16z_{1,0}-6z_{0,0})=0  
\\F_4: &  -z_{0,0}z_{2,2}-z_{0,1}z_{2,1}+4z^2_{1,1}+8z_{1,0}z_{1,2}+4z_{2,0}(z_{0,2}-z_{1,2})+z_{3,0}z_{0,2}=0  
\eer\eeq

\bthe\label{The.Case.n=11} {\bf 1.} The parameter space has two
irreducible components, $X$ and $Y$, both of dimension five, as a set $\m
_{11}=X\cup Y$.
\\{\bf 2.} The component  $X=\{z_{0,0}=0=z_{1,0}=z_{2,0}=4z^2_{1,1}-z_{0,1}z_{2,1}+z_{3,0}z_{0,2}\}\sset\P^9$ enters with generic multiplicity 2.
The singular locus of (reduced) $X$ is
$Sing(X)=\{z_{0,0}=0=z_{1,0}=z_{2,0}=z_{1,1}=z_{0,1}=z_{2,1}=z_{3,0}=z_{0,2}\}=\P^1_{z_{1,2}z_{2,2}}$.

$X$ admits the algebraic cell structure: $\C^5\cup
\C^4\cup\C^3\cup\C^2\cup\C^1\cup\C^0$. In particular, its odd
cohomologies vanish, while all the even cohomologies are $\Z$.
\\{\bf 3.} The component $Y$ is reduced, it is the topological closure of the affine part of $\m _{11}$ in $\C^9=\{z_{0,0}\neq0\}\subset\Pr^9$.
(Thus, the defining equations of $Y\cap \C^9$ are obtained from
the equations above by setting $z_{0,0}=1$.) The affine part
$Y\cap\{z_{0,0}\neq0\}$ is smooth. The intersection with the
infinite hyperplane, $Y\cap\{z_{0,0}=0\}$, is defined, as a set,
by
$\{z_{0,0}=0=z_{1,0}=z_{3,0}=z_{0,1}=z^2_{1,1}+z_{0,2}(z_{0,2}-z_{1,2})\}$.

The affine part, $Y\cap\{z_{0,0}\neq0\}\sset\C^9$, is isomorphic
to $\C^4$-bundle over the line with two punctures,
$\C^1_{z_{1,0}}\smin\{(z_{1,0}+1)(3z_{1,0}-1)=0\}$.

$H_{2i}(Y,\Z)=\Z$, for $0\le 2i\le 10$ and $H_9(Y,\Z)=\Z^{\oplus
2}$, all the other cohomologies vanish. \ethe \noindent\bpr {\bf
1.} Consider the part of $\m_{11}$ at infinity,
$\m_{11}\cap\{z_{0,0}=0\}$. We get (omitting multiplicities)
$z_{0,0}=0=z_{1,0}=z_{2,0}z_{0,1}$. Hence this part splits, by
direct check we get two components: \beq
\m_{11}\cap\{z_{0,0}=0\}=\Big\{z_{0,0}=0=z_{1,0}=z_{2,0}=4z^2_{1,1}-z_{0,1}z_{2,1}+z_{3,0}z_{0,2}\Big\}\cup
\Big\{z_{0,0}=0=z_{1,0}=z_{3,0}=z_{0,1}=z^2_{1,1}+z_{2,0}(z_{0,2}-z_{1,2})\Big\}
\eeq Note that neither of them lies inside the other, e.g. they
are distinguished by $z_{2,0}=0$ and $z_{0,1}=0$. By direct check,
the first component is of codimension 4, the second is of
codimension 5. Note that the whole space, $\m_{11}$ is
 defined by four equations, therefore at each point of $\m_{11}$ the (local) codimension is at most 4.
 Therefore the first component is an honest component of $\m_{11}$ (we call it $X$),
while the second component belongs to the intersection of $Y$ with
the infinite hyperplane $z_{0,0}=0$.

{\bf 2.} To check the generic multiplicity of $X$, fix some
generic values of
$z_{0,2},z_{1,2},z_{2,2},z_{3,0},z_{0,1},z_{1,1}$ and let
$z_{0,0},z_{1,0},z_{2,0}$ vary near zero, while $z_{2,1}$ varies
near a root of $4z^2_{1,1}-z_{0,1}z_{2,1}+z_{3,0}z_{0,2}=0$. This
corresponds to the transverse intersection of the generic point of
$X$ by a linear space of the complementary dimension. The generic
multiplicity of $X$ is the multiplicity of this local
intersection, i.e. (algebraically) the length of the Artinian ring
$\C\{z_{0,0},z_{1,0},z_{2,0},z_{2,1}\}/<F_1,F_2,F_3,F_4>$. Note
that for the generic values of
$z_{0,2},z_{1,2},z_{2,2},z_{3,0},z_{1,1}$, the variable $z_{2,1}$
enters linearly in equation $F_4$, hence can be eliminated.
Similarly, the variable $z_{0,0}$ can be eliminated using $F_3$,
while $z_{2,0}$ can be eliminated using $F_2$. Thus the ring
$\C\{z_{0,0},z_{1,0},z_{2,0},z_{2,1}\}/<F_1,F_2,F_3,F_4>$ is
isomorphic to $\C\{z_{1,0}\}/Q$, where the expansion of $Q$ in
$z_{1,0}$ begins from a quadratic term. Hence the intersection
multiplicity (=the length of this ring) is two.

To understand the cell structure, consider the affine part,
$X\cap\{z_{3,0}\neq0\}$, and the part at infinity,
$X\cap\{z_{3,0}=0\}$. By direct check:
$X\cap\{z_{3,0}\neq0\}\approx
\C^5_{z_{0,1}z_{1,1}z_{2,1}z_{1,2}z_{2,2}}$. Continue "to cut" the
infinite part: \beq
X\cap\{z_{3,0}=0\}=\underbrace{X\cap\{z_{3,0}=0=z_{0,1}\}}_{\P^3_{z_{1,1}z_{2,1}z_{1,2}z_{2,2}}}\coprod
\underbrace{X\cap\{z_{3,0}=0,\
z_{0,1}\neq0\}}_{\C^4_{z_{0,1}z_{1,1}z_{1,2}z_{2,2}}} \eeq Finally
we use the standard cell decomposition
$\P^3=\C^3\coprod\C^2\coprod\C^1\coprod\C^0$. This gives the
algebraic cell decomposition of $X$.

{\bf 3.} The defining equations of the affine part of
$Y|_{z_{0,0}\neq0}$ are: \beq\ber
z_{2,0}=3z_{1,0}^2,\quad  z_{2,1}=7z_{1,0}z_{1,1}-z_{2,0}z_{0,1},\quad z_{3,0}(2z_{1,0}-1+z_{2,0})+3z^2_{1,0}(16z_{1,0}-6)=0,\\
z_{2,2}=-z_{0,1}z_{2,1}+4z^2_{1,1}+8z_{1,0}z_{1,2}+4z_{2,0}(z_{0,2}-z_{1,2})+z_{3,0}z_{0,2}
\eer\eeq Therefore $Y$ projects isomorphically onto
$\{z_{3,0}(2z_{1,0}-1+3z_{1,0}^2)+3z^2_{1,0}(16z_{1,0}-6)=0\}\sset\C^6_{z_{1,0}z_{3,0}z_{0,1}z_{1,1}z_{0,2}z_{1,2}}$.
Note that the defining equation does not contain variables
$z_{0,1}z_{1,1}z_{0,2}z_{1,2}$, so this is a $\C^4$ bundle over
the curve
$\{z_{3,0}(2z_{1,0}-1+3z_{1,0}^2)+3z^2_{1,0}(16z_{1,0}-6)=0\}\sset\C^2_{z_{1,0}z_{3,0}}$.
Finally, this curve projects isomorphically onto $\C^1_{z_{1,0}}$
outside the locus $(3z_{1,0}-1)(z_{1,0}+1)=0$, i.e. with two
points punctured.

Finally, we use \S\ref{Sec.Exact.Sequence.of.Pair} for the pair:
$Y|_{z_{0,0}=0}\sset Y$: $\cdots\to H_i(Y|_{z_{0,0}=0},\Z)\to
H_i(Y,\Z)\to H_i(Y,Y|_{z_{0,0}=0},\Z)\to\cdots$.
 Note that $Y\smin Y|_{z_{0,0}=0}$ is smooth of real dimension 10, thus by proposition \ref{Thm.Cohom.Non.Compact.Spaces}:
 $H_i(Y,Y|_{z_{0,0}=0},\Z)=H^{10-i}(Y\smin Y|_{z_{0,0}=0},\Z)$. As established above, the homotopy type of
 $Y\smin Y|_{z_{0,0}=0}$ is that of $\C^1$ with two punctured points. Therefore
$H_i(Y|_{z_{0,0}=0},\Z)\isom H_i(Y,\Z)$ for $i<8$ and
\[
0\to H_9(Y,\Z)\to H_1(\C^1\smin\{2pts\},\Z)\to
H_8(Y|_{z_{0,0}=0},\Z)\to H_8(Y,\Z)\to0
\]
As shown above, $Y|_{z_{0,0}=0}$ is a singular quadric in $\P^5$,
and its cell structure is: $\C^4\cup\C^3\cup\C^2\cup\C^1\cup\C^0$.
In particular, all its even cohomologies are $\Z$, while all the
odd cohomologies vanish. Besides note that $H_8(Y,\Z)$ contains
(at least) one factor of $\Z$, being a projective hypersurface.
Therefore, from $H_8(Y|_{z_{0,0}=0},\Z)\to H_8(Y,\Z)\to0$, we get:
$H_8(Y,\Z)=\Z$ and the last map is an isomorphism. Thus: $0\to
H_9(Y,\Z)\to H_1(\C^1\smin\{2pts\},\Z)\to0$, proving the
statement. \epr

\beR \noindent {\bf About the first component of $\m_{11}$:} the
component $X$ of $\m_{11}$  contains the $\mathbb{Z}$-graded
algebra $\mathrm{g}$ of type $\mathrm{m}_{0,1}(11)$ defined in
\cite{mill:04} (p. 268) which would correspond to
$z_3=z_4=z_5=z_7=z_8=z_9=0$ but $z_6\ne 0.$ Besides, $X$ also
contains   algebras corresponding to $z_7=z_8=z_9=z_6=0.$ The
corresponding deformation cocycle is \[\Psi(e_i,e_j) = \alpha_{ij}
e_{i+j+1}\] where $i,j>1$, $\alpha_{ij}$ are scalars. Besides,
$\Psi$ vanishes on the other pairs of vectors from the standard
basis $\{e_1, e_2, \ldots, e_{11}\}$ of $\mathrm{m}_0(11)$.
Re-naming basis elements: $f_1=e_1$ and $f_{i+1} = e_{i}$ where
$i=2,...,11$ we obtain an $\mathbb{N}$-graded basis for
$\mathrm{g}$ corresponding to $\Psi$. As would follow from results
of Section 3, $\mathrm{g}$ should be of one of the following
types: $\mathrm{m}^3_{0,2}(12)$ (the second central extension of
$\mathrm{m}^3_0(10)$), $\mathrm{m}^3_{0,3}(12)$ (the third central
extension of $\mathrm{m}^3_0(9)$), $\mathrm{m}^3_{0,5}(12)$ (the
fifth central extension of $\mathrm{m}^3_0(7)$),
$\mathrm{m}^3_{0,6}(12)$ (the sixth central extension of
$\mathrm{m}^3_0(6)$).

There are also algebras in $X$ corresponding to
$z_3=z_4=z_5=z_6=0.$ In this case,  the deformation cocycle is
$\Psi(e_i, e_j) = \alpha_{ij} e_{i+j+2}$ for $i,j>1$, and it
equals to zero on the remaining vector-pairs. Changing  notation
of the basis elements of $\mathrm{m}_0(11)$: $f_1=e_1, f_{i+2} =
e_{i},$ $i=2,..., 11$, we obtain  an $\mathbb{N}$-graded basis for
$\mathrm{g}$ corresponding to $\Psi$. Thus, $\mathrm{g}$ is
obtained by taking one-dimensional central extensions of
$\mathrm{m}^4_{0}(8)$. Therefore, $X$ contains algebras of the
following types: $\mathrm{m}^4_{0,1}(13)$ (one-dimensional central
extension of $\mathrm{m}^4_0(12)$), $\mathrm{m}^4_{0,3}(13)$ (the
third central extension of $\mathrm{m}^4_{0}(10)$,
$\mathrm{m}^4_{0,5}(13)$ (the fifth central extension of
$\mathrm{m}^4_0(8)$). \eeR

\section{Structure of $\mathbb{N}$-graded  filiform Lie algebras}
\label{algebra}

The classification of nilpotent Lie algebras is a difficult
problem  widely discussed in literature. Nilpotent Lie algebras up
to dimension 5 are well-known. In  \cite{graaf:12}  the authors
gave a full classification of  6-dimensional nilpotent Lie
algebras over arbitrary fields.  In higher dimensions, there are
infinite families of pairwise nonisomorphic nilpotent Lie
algebras. In dimension 7 each infinite family can be parameterized
by a single parameter. Many papers on classification of
7-dimensional nilpotent Lie algebras have been published but  the
most complete list of such Lie algebras was obtained by Ming-Peng
Gong (see \cite{gong:98}).

In this section we are concerned  with  those nilpotent Lie
algebras whose nil-index is $n-1$ for a given dimension $n$ over
\emph{an algebraically closed field $F$ of zero characteristic}.
Such Lie algebras will be called \emph{filiform}. An
infinite-dimensional analog of a filiform Lie algebra is a
so-called Lie algebra of \emph{maximal class} (or of coclass 1).
Namely, a residually nilpotent Lie algebra is called a Lie algebra
of \emph{maximal class} (or coclass 1) if $\sum_{i\ge 1} ( \dim
\mathrm{g}^i/\mathrm{g}^{i+1} -1) =1$ where $\{ \mathrm{g}^i\}$ is
the lower central series of $\mathrm{g}$.  In \cite{ver:70} M.
Vergne has shown that an arbitrary filiform Lie algebra is
isomorphic to some deformation of the graded filiform Lie algebra
$\mathrm{m}_0(n)$ defined by its basis $e_1,\ldots, e_n$ and
nontrivial Lie products: $[e_1, e_i] = e_{i+1},$ $i=2,\ldots,n-1$.

\begin{Example} Let $\mathrm{m}_0$ be a linear space with a basis
$\{e_1, e_2, \ldots \}$. Define the Lie product on $\mathrm{m}_0$
by $[e_1, e_i] = e_{i+1}$ for $i>1$ and the other products are
zero. Note that we can introduce two types of
$\mathbb{N}$-gradings:

\noindent\emph{Type 1:} $\mathrm{m}_0 = \oplus_{i\ge 1} L_i$
 where $L_1 = \text{span}\{e_1, e_2\} $ and $L_i =
\text{span} \{ e_{i+1}\}$.

\noindent\emph{Type 2:} $\mathrm{m}_0 = \oplus_{i\ge 1} {\tilde
L}_i$
 where ${\tilde L}_1 = \text{span}\{e_1\} $ and ${\tilde L}_i =
\text{span} \{ e_{i}\}$.
\end{Example}

In the case of infinite-dimensional $\mathbb{N}$-graded Lie
algebras of maximal class, Vergne proved the following:

\begin{Theorem}\label{Vergne}  Let $L = \oplus_{i\in \mathbb{N}} L_i$ be an
infinite-dimensional $\mathbb{N}$-graded Lie algebra of maximal
class and suppose $L=\langle L_1  \rangle.$ Then  $L\cong
\mathrm{m}_0$ (with Type 1 grading).
\end{Theorem}

\noindent By taking quotients of $\mathrm{m}_0$ we obtain
finite-dimensional  filiform Lie algebras $\mathrm{m}_0(n) =
\mathrm{m}_0/I_n$ where $I_n = \text{span}\{e_{n+1}, e_{n+2},
\ldots\}$.   We next introduce other important examples of
infinite-dimensional  Lie algebras of maximal class.

\begin{Example}
The Lie algebra $\mathrm{m}_2$ is defined by its basis $\{e_1,
e_2, \ldots \}$ with multiplication table as follows
\[ [e_1, e_i] = e_{i+1}, \, i\ge 2, \, [e_2, e_i]= e_{i+2},\, i\ge
3,\] and the remaining products are all zero.
\end{Example}

\begin{Example}
The  Lie algebra $W$ (the Witt algebra) is defined by its basis
$\{e_1, e_2, \ldots \}$ with multiplication table as follows
\[ [e_i, e_j] = (j-i) e_{i+j}, \,\, i,j \ge 1.\]
\end{Example}

In \cite{fialowski:83}  the classification of $\mathbb{N}$-graded
Lie algebras of maximal class $L = \oplus_{i\in \mathbb{N}} L_i$
generated by $L_1, L_2$ was obtained. Namely, the following
theorem holds.
\par\medskip
\begin{Theorem}\label{Zelmanov}  Let $L = \oplus_{i\in \mathbb{N}} L_i$ be an
infinite-dimensional $\mathbb{N}$-graded Lie algebra of maximal
class and suppose $L=\langle L_1, L_2 \rangle.$ Then one of the
following holds:

(1) $L\cong \mathrm{m}_0$;

(2) $L\cong \mathrm{m}_2$;

(3) $L\cong W$.
\end{Theorem}
\par\medskip
(Note that this result was also obtained 14 years later in [SZ],
and it also follows from Theorem 5.17 in \cite{mill:04} 2004.)

\par\medskip

Let us now consider $\mathbb{N}$-graded Lie algebras of maximal
class that are generated by  graded components of degrees 1 and
$q$ where $q>2$. Hence,
\[ \mathrm{g} = \bigoplus^{\infty}_{i=1,q} \mathrm{g}_i. \]
 Below are given some examples of such algebras.

\begin{Example} The Lie algebra $\mathrm{m}^q_0$ is defined by its basis
$e_1, e_q,\ldots $ with multiplication table as follows
\[ [e_1, e_i] =e_{i+1}, \,\, i\ge q \]
and the remaining products are zero. The basis as above will be
called  the {\bf  standard} basis for $\mathrm{m}^q_0$.
\end{Example}

By taking quotients of $\mathrm{m}^q_0$ we obtain
finite-dimensional filiform Lie algebras $\mathrm{m}^q_0(n) =
\mathrm{m}^q_0/I_n$ where $I_n = \text{span}\{e_{n+1}, e_{n+2},
\ldots\}$ also generated by components of degrees 1 and  $q$.

\begin{Example} The Lie algebra $\mathrm{m}_q$ has the basis $e_1,
e_q,\ldots $ and the following multiplication table:
\begin{align*}
 &[e_1, e_i] = e_{i+1}, \,\, i\ge q, \\
 &[e_q, e_i] = e_{q+i}, \,\, i\ge q+1,
\end{align*}  and the other products are zero.
\end{Example}

\begin{Example} The Lie algebra $W^q$ is given by its basis $e_1,
e_q,\ldots $ with the following multiplication table:
\[ [e_i, e_j] = (j-i) e_{i+j}\]
and the remaining products are all zero.
\end{Example}

Notice that $W^q$ is a \emph{nonsolvable} Lie algebra of maximal
class. It is not known yet whether there are nonsolvable Lie
algebras of maximal class other than algebras described in the
preceding example. The isomorphism classes of solvable Lie
algebras of maximal class were given in \cite{B:74} (see also
\cite{CL:10}).

\par\medskip
\noindent Here is the main conjecture:

\par\medskip
\noindent{\bf Conjecture.} Let $\mathrm{g}$ be an
$\mathbb{N}$-graded  Lie algebra of maximal class generated by
graded components of degrees 1 and $q$. Then $\mathrm{g}$ is
isomorphic (as a graded algebra) to one of the following three
algebras: $\mathrm{m}^q_0$, $\mathrm{m}_q$, $W^q$.
\par\medskip
Later we will see that this conjecture is actually equivalent to
the conjecture from \cite{mill:09} p. 190.

For  $q>2$, we show the following:

\begin{Theorem}\label{MainTheorem}
Let $\mathrm{g}$  be an $\mathbb{N}$-graded Lie algebra of maximal
class generated by nonzero graded components $\mathrm{g}_1$ and
$\mathrm{g}_q$ where $q>2$, and let  $\mathrm{g}_{q+2}\ne \{0\}$.
If
\[ [\mathrm{g}_q, \mathrm{g}_{q+1}] =\ldots=[\mathrm{g}_{2q},
\mathrm{g}_{2q+1}]=0, \]
 then $\mathrm{g}\cong
\mathrm{m}^q_0$.
\end{Theorem}

In some sense this result is similar to Theorem \ref{Vergne}. If
$\mathrm{g}$ is generated by two graded components as above, then
under some technical condition there can  only be one isomorphism
type.

Besides, we prove the conjecture for $q=3$ using some general
results on central extensions of $\mathrm{m}^q_0(n)$ obtained in
subsections 3.1 and 3.2.
\par\medskip
\par\medskip

\subsection{Central extensions}

Let $L$ be a Lie algebra and $V$ a vector space  with a
skew-symmetric bilinear form $\theta: L\times L \mapsto V$, i.e.
$\theta(x,x)=0$ for  all $x\in L$. Then $\theta$ as above
satisfying
\[ \theta([x,y],z) + \theta([z,x], y) + \theta([y,z], x) = 0\]
where $x,y,z\in L$ is said to be a \emph{cocycle}. If
$\theta:L\times L\mapsto V$ is a cocycle, then $L_{\theta} = L
\oplus V$ with the product defined by
\[ [x+v, y+w]' = [x,y] + \theta(x,y) \] is a Lie algebra.
Then $L_{\theta}$ is said to be a \emph{central extension} of $L$
by $V$. Note that $V$ is central in $L_{\theta}$. If both $L$ and
$L_{\theta} = L \oplus V$ are filiform, then $\theta \ne 0$.
Otherwise, $L^2_{\theta} = L^2$ and
\[ L_{\theta}/L^2_{\theta} = (L\oplus V)/L^2 = L/L^2 \oplus V.\]
Then
\[ \dim L_{\theta}/L^2_{\theta} = \dim L/L^2 + \dim V \ge 3, \]
since $\dim L/L^2=2$ (this fact holds for any filiform Lie
algebra). Therefore, $L_{\theta}$ cannot be filiform, a
contradiction. Furthermore,  if $L_{\theta}=L\oplus V$ is a
one-dimensional  filiform central extension of a filiform $L$,
i.e. $\dim V =1$,  then $L_{\theta}$ is generated by $L$. Indeed,
since $\dim\,V = 1$, $V=\text{span}\,\{w\}$, $w\ne 0$. As noted
above, $\theta \ne 0$, i.e. there are two $x, y \in L$ such that
$\theta(x,y)\ne 0$. Thus, $[x,y]' = [x,y]+\theta(x,y)=
[x,y]+\alpha w$ for some $\alpha\ne 0$. Hence, $w= \alpha^{-1}
[x,y]' - \alpha^{-1} [x,y] \in \langle L\rangle,$ and, therefore,
$V\subseteq \langle L\rangle$.
\par\medskip

Let $\mathrm{g}$ be an $\mathbb{N}$-graded filiform Lie algebra
generated by nonzero $\mathrm{g}_1$ and $\mathrm{g}_q$, $q>2$.
Then
 $\mathrm{g} = \mathrm{g}_1\oplus \mathrm{g}_q \oplus\ldots
\oplus \mathrm{g}_n$ for some $n$. Without any loss of generality
we assume that $\mathrm{g}_n\ne \{0\}$, otherwise, we discard it.

\begin{Lemma}
Let $\mathrm{g}$ be an $\mathbb{N}$-graded filiform Lie algebra
generated by nonzero  $\mathrm{g}_1$ and $\mathrm{g}_q$.
Additionally, assume that $\mathrm{g}_{2+q}\ne\{0\}.$ Then every
$\mathrm{g}_i$, $i=1, q,\ldots, n$ is a nonzero component of
dimension one.
\end{Lemma}

\begin{proof}
We first want to prove that if $\mathrm{g}_i \ne \{0\}$, then
$\dim\mathrm{g}_i = 1$.  Since $\mathrm{g} =\langle \mathrm{g}_1,
\mathrm{g}_q \rangle$,   $\mathrm{g}_{2+q} = [\mathrm{g}_1,
[\mathrm{g}_1, \mathrm{g}_{q} ]] \ne 0 $.  Hence, $[\mathrm{g}_1,
\mathrm{g}_{q}]\ne 0$ and  $\mathrm{g}_{1+q}=[\mathrm{g}_1,
\mathrm{g}_{q} ]\ne \{0\}$. Then we can write $\mathrm{g}$ as
\[ \mathrm{g} =
\mathrm{g}_1\oplus\mathrm{g}_q\oplus\mathrm{g}_{q+1}\oplus\mathrm{g}_{q+2}\oplus
\mathrm{g}_{i_1}\oplus\ldots\oplus\mathrm{g}_{i_s} \] where
$\mathrm{g}_{i_1},\ldots,\mathrm{g}_{i_s}$ are the remaining
nonzero graded components. Hence, $\dim\mathrm{g} \ge 4+s$ (the
total number of nonzero components). Since $\mathrm{{g}}$ is
filiform, its nil-index $m = \dim\,\mathrm{{g}} -1\ge 3+s.$
Directly computing components of the lower central series of
$\mathrm{{g}}$ we obtain the following:
\begin{align*}
& \mathrm{{g}}^2 \subseteq
\mathrm{{g}}_{q+1}+\ldots+\mathrm{g}_{i_s},\\
& \mathrm{{g}}^3 \subseteq
\mathrm{{g}}_{q+2}+\ldots+\mathrm{g}_{i_s},\\
& \mathrm{{g}}^4 \subseteq
\mathrm{{g}}_{i_1}+\ldots+\mathrm{g}_{i_s},\\
& \ldots\\
&\mathrm{{g}}^{3+r} \subseteq
\mathrm{{g}}_{i_r}+\ldots +\mathrm{{g}}_{i_s},\\
& \ldots \\
&\mathrm{{g}}^{3+s-1} \subseteq
\mathrm{{g}}_{i_{s-1}}+\mathrm{{g}}_{i_s},\\
&\mathrm{{g}}^{3+s} \subseteq \mathrm{{g}}_{i_s},\\
&\mathrm{{g}}^{4+s}=\{0\}.
\end{align*}
This means that nil-index $m\le 3+s.$ Therefore, $m=\dim
\mathrm{{g}}-1 = 3+s$, and $\dim \mathrm{{g}} = 4+s.$ Since there
are exactly $4+s$  nonzero graded components, each component must
be one-dimensional. Since $\dim\,\mathrm{g}/\mathrm{g}^2 =2$ and
$\dim\,\mathrm{g}^i/\mathrm{g}^{i+1} =1,$ $i\ge 2$,  all
inclusions above become equalities.

We next show that there is no `gap' in the grading from $q+1$ to
$n$. This means that all $\mathrm{g}_i$, $i=q+1,\ldots, n$ must be
nonzero.  Assume the contrary, i.e. there exists $s$, $q< s< n$
such that $\mathrm{g}_s = \{0\}$. Let $s$ be the smallest number
satisfying this condition.  Clearly, $s> 2+q$. Let
$\mathrm{g}_{s+t}$, $t\ge 1$, $s+t \le n$ be the first nonzero
component following $\mathrm{g}_{s-1}$. Consider
$\mathrm{\tilde{g}} = \mathrm{g}/J$ where $J=\bigoplus_{j>s+t}
\mathrm{g}_j$ is the ideal of $\mathrm{g}$. Then
$\mathrm{\tilde{g}}$ is also filiform, and
\[ \mathrm{\tilde{g}} = \mathrm{\tilde{g}}_1 \oplus
\mathrm{\tilde{g}}_q\oplus\ldots\oplus
\mathrm{\tilde{g}}_{s-1}\oplus \mathrm{\tilde{g}}_{s+t} \] where
$\mathrm{\tilde{g}}_i = \mathrm{g}_i + J $, $\dim\,
\mathrm{\tilde{g}}_i = \dim\, \mathrm{{g}}_i=1$. Besides,
$\mathrm{\tilde{g}}$ is generated by $\mathrm{\tilde{g}}_1$ and
$\mathrm{\tilde{g}}_q$ too.   We next choose a basis: ${e}_1,
{e}_q,\ldots, {e}_{s-1}, {e}_{s+t}$ such that $[{e}_1, {e}_q] =
{e}_{q+1}$, and $[{e}_1, {e}_{q+1}] = {e}_{q+2}$. Since
$\mathrm{\tilde{g}}_s = \{0\}$, $[{e}_1, {e}_{s-1}]=0.$ It is
known that a filiform Lie algebra $\mathrm{\tilde{g}}$ has an
\emph{adapted} basis: $f_1, f_2,\ldots, f_k$,
$k=s-q+2=\dim\mathrm{\tilde{g}}$ such that $[f_1, f_i] = f_{i+1}$,
$i=2, \ldots, k-1$, and $[f_i, f_j]\in
\text{span}\{f_{i+j},\ldots, f_k\}$. Moreover,

\begin{align*}
&\mathrm{\tilde{g}}/\mathrm{\tilde{g}}^2 = \text{span}\{f_1,
f_2\}+ \mathrm{\tilde{g}}^2=\text{span}\{e_1, e_q\}+
\mathrm{\tilde{g}}^2,\\
& \ldots\\
& \mathrm{\tilde{g}}^i/\mathrm{\tilde{g}}^{i+1} =
\text{span}\{f_{i+1}\}+ \mathrm{\tilde{g}}^{i+1}=\text{span}\{
e_{q+i-1}\}+
\mathrm{\tilde{g}}^{i+1},\\
&\ldots\\
& \mathrm{\tilde{g}}^{s-q}/\mathrm{\tilde{g}}^{s-q+1}
=\text{span}\{f_{s-q+1}\}+
\mathrm{\tilde{g}}^{s-q+1}=\text{span}\{ e_{s-1}\}+
\mathrm{\tilde{g}}^{s-q+1},\\
& \mathrm{\tilde{g}}^{s-q+1}/\mathrm{\tilde{g}}^{s-q+2}
=\text{span}\{f_{s-q+2}\} = \text{span} \{ e_{s+t} \}
\end{align*}
Therefore, $e_1 = \lambda_1 f_1 + \lambda_2 f_2 + h$ where $h\in
\text{span}\{f_3,\ldots,f_k\}$ and $e_{s-1} = \mu f_{k-1} + \beta
f_k$, $\mu\ne 0$, $k=s-q+2$. Then,
\[ 0= [e_1, e_{s-1}] = [ \lambda_1 f_1+\lambda_2 f_2 +h, \mu
f_{k-1}+\beta f_k] =\lambda_1\mu [f_1, f_{k-1}]= \lambda_1 \mu
f_k.\] Since $\mu\ne 0$ we have that $\lambda_1=0$. Hence,
$e_1=\lambda_2 f_2 +h$. Write $e_{q+1} =\gamma f_3 + h'$,
$\gamma\ne 0$, $h'\in\text{span}\{f_4,\ldots, f_k\}$,  and
$e_{q+2} =\delta f_4 + h''$, $\delta\ne 0$,
$h''\in\text{span}\{f_5,\ldots, f_k\}$. Therefore,
\[ e_{q+2}= [e_1, e_{q+1}] = [\lambda_2 f_2 +h, \gamma f_3 + h'] =
\lambda_2 \gamma f_5 + \bar h,\] where $\bar h \in
\text{span}\{f_6, \ldots, f_k\}.$ Comparing with $e_{q+2} =\delta
f_4 + h''$ we obtain that $\delta=0$, a contradiction. This means
that there cannot be any `gap' in the grading of $\mathrm{g}$. The
proof is complete.
\end{proof}

\begin{Lemma}\label{techlemma2}
Let $\mathrm{g}$ be an $\mathbb{N}$-graded filiform Lie algebra
generated by nonzero  $\mathrm{g}_1,$ $\mathrm{g}_q$, and let
$\mathrm{g}_{2+q}\ne\{0\}.$ Then there is a basis for
$\mathrm{g}:$ $e_1, e_q, \ldots, e_n$ such that $\mathrm{g}_i =
\text{span}\{ e_i \}$ and $[e_1, e_i] = e_{i+1},$ $i=q,\ldots,
n-1$.
\end{Lemma}

\begin{proof}
As follows from the previous lemma, each component is of dimension
one. Therefore, it suffices to show that $[ \mathrm{g}_1,
\mathrm{g}_i ] \ne 0$ for any $i=1,q,\ldots, n-1$. We know that
$[\mathrm{g}_1, \mathrm{g}_q] \ne 0,$  $[\mathrm{g}_1,
\mathrm{g}_{q+1}] \ne 0.$ Assume that there exists $i$, $n> i >
q+1$ such that $[\mathrm{g}_1, \mathrm{g}_i] = 0.$  Consider
$\mathrm{\tilde g} = \mathrm{g}/J $ where $J = \bigoplus_{j> i+1}
\mathrm{g}_j.$ Then
\[ \mathrm{\tilde{g}} = \mathrm{\tilde{g}}_1 \oplus
\mathrm{\tilde{g}}_q\oplus\ldots\oplus
\mathrm{\tilde{g}}_{i}\oplus \mathrm{\tilde{g}}_{i+1} \] where
$\mathrm{\tilde{g}}_l = \mathrm{g}_l + J,$ $l=1,q,\ldots, i+1$.
Then $\mathrm{{\tilde g}}^{i-q+1} =\mathrm{{\tilde g}}_{i}\oplus
\mathrm{{\tilde g}}_{i+1}$ and $\mathrm{{\tilde g}}^{i-q+2} =
[\mathrm{\tilde{g}}, \mathrm{{\tilde g}}_{i}\oplus\mathrm{{\tilde
g}}_{i+1}] = \{0\}.$ This means that $ \dim
\mathrm{\tilde{g}}^{i-q+1}/\mathrm{\tilde{g}}^{i-q+2} = 2$. This
contradicts to the fact that  $\mathrm{{\tilde g}}$ is filiform.
Therefore, $[\mathrm{g}_1, \mathrm{g}_i ] \ne 0$ for any
$i=1,q,\ldots, n-1$. It is now easy to see that we can choose a
required basis for $\mathrm{g}$.
\end{proof}

The following corollaries are immediate consequences of the above
lemmas.

\begin{Corollary}\label{techlemma1}
Let $\mathrm{g}$ be an $\mathbb{N}$-graded filiform Lie algebra
generated by nonzero  $\mathrm{g}_1,$ $\mathrm{g}_q$, and let
$\mathrm{g}_{2+q}\ne\{0\}.$ If $n < 2q+1$, then $\mathrm{g} \cong
\mathrm{m}^q_0(n)$.

\end{Corollary}

\begin{Corollary}\label{corollary1}
Let $\mathrm{g}$ be an $\mathbb{N}$-graded Lie algebra of maximal
class generated by both graded components $\mathrm{g}_1$ and
$\mathrm{g}_q$, and let $\mathrm{g}_{q+2}\ne \{0\}.$ Then each
graded component is one-dimensional. Moreover, there exists a
basis for $\mathrm{g}:$ $e_1, e_q, \ldots $ such that
$\mathrm{g}_i = \text{span}\{e_i\}$, and $[e_1, e_i] =e_{i+1}$,
$i>1$.

\end{Corollary}

\begin{Definition}
Let $\mathrm{g}$ be an $\mathbb{N}$-graded filiform Lie algebra as
above. Any basis $\{ e_1, e_q, \ldots, e_n\}$ of $\mathrm{g}$
satisfying $\mathrm{g}_i = \text{span}\{e_i\},$ and $[e_1, e_i] =
e_{i+1}$, $i>1$,  will be called {\bf canonical}.
\end{Definition}

\subsection{Central extensions of $\mathrm{m}^q_0(n)$}
In this section we discuss  one-dimensional $\mathbb{N}$-graded
filiform central extensions of $\mathrm{m}^q_0(n)$. The following
lemma is similar to Corollary 5.3 \cite{mill:04}. As was noted
earlier if $\mathrm{g}$ is an $\mathbb{N}$-graded one-dimensional
filiform central extension of $\mathrm{m}^q_0(n)$, then
$\mathrm{g}$ is generated by $\mathrm{g}_1$ and $\mathrm{g}_q$
since $\mathrm{m}^q_0(n)=\langle \mathrm{g}_1, \mathrm{g}_q
\rangle$ (see the beginning of subsection 3.1).

\begin{Lemma}\label{onedimextensions}
Let $\mathrm{g}$ be a one-dimensional $\mathbb{N}$-graded filiform
central extension of $\mathrm{m}^q_0(n)$. Then

{\bf 1.} If $n=2k+1$, then $\mathrm{g}\cong \mathrm{m}^q_0(2k+2).$

{\bf 2.} If $n=2k$, then either $\mathrm{g}\cong
\mathrm{m}^q_0(2k+1)$ or $\mathrm{g}\cong
\mathrm{m}^q_{0,1}(2k+1)$ defined by the basis $e_1, e_q, \ldots,
e_{2k}, e_{2k+1}$ and structure relations:
\[ [e_1, e_i] = e_{i+1},\,\, i=q,\ldots,2k \,\,\text{and}\,\,
[e_r,  e_{2k+1-r}] = (-1)^{r-k} e_{2k+1},\,\,r=q,\ldots,k.\]

\end{Lemma}

\begin{proof}

First we consider the case of odd $n=2k+1$. If $k<q$, then $n=2k+1
< 2q+1$ and $n+1=2k+2 < 2q+1$.  By Corollary \ref{techlemma1},
$\mathrm{g}\cong\mathrm{m}^q_0(2k+2)$. Let now $k\ge q$. Let us
choose the  \emph{standard}  basis for $\mathrm{m}^q_0(2k+1)$:
\[e_1, e_q,\ldots, e_{2k+1}\] where $[e_i, e_j] = \lambda_{ij}
e_{i+j},$ $\lambda_{1i}= 1$, $i\ge q$ and $\lambda_{ij} = 0$,
$i,j\ge q$.   Let $\mathrm{g}$ denote an $\mathbb{N}$-graded
one-dimensional filiform central extension of
$\mathrm{m}^q_0(2k+1)$. By Lemma \ref{techlemma2}, the standard
basis can be extended to the following  canonical  basis $e_1,
e_q,\ldots, e_{2k+1}, e_{2k+2}$ of $\mathrm{g}$  such that
\begin{align*}
&[e_i, e_{2k+2-i}] = \lambda_{i,2k+2-i}
e_{2k+2},\,\,{i=q,\ldots,k}.
\end{align*}
Note that if $i+j< 2k+2$ then the products $[e_i, e_j]$ are
exactly the same as in $\mathrm{m}^q_0(2k+1)$. Let us find unknown
structure constants $\lambda_{i,2k+2-i},$ $i=q,\ldots, k.$ We know
that $J(e_1, e_r, e_{2k+1-r})=0$ where $r=q,\ldots, k$ and
$J(\,\,\,)$ is the Jacobian. This equation can be re-written in
terms of $\lambda$'s as follows:
\begin{equation}\label{eq}
\lambda_{1r} \lambda_{1+r, 2k+1-r} + \lambda_{r, 2k+1-r}
\lambda_{2k+1, 1}+ \lambda_{2k+1-r, 1} \lambda_{2k+2-r, r} = 0
\end{equation}
Note that $\lambda_{1r} = 1,$ $\lambda_{2k+1-r, 1} = -1$, and
$\lambda_{r, 2k+1-r}=0.$ Therefore, \eqref{eq} becomes
\begin{equation}
\lambda_{1+r, 2k+1-r} + \lambda_{r, 2k+2-r} =0,\,\, r=q,\ldots,k,
\end{equation}
and $\lambda_{k+1, k+1} = 0$. Clearly, this system has a unique
solution: $\lambda_{r, 2k+2-r} = 0,$ $r=q,\ldots, k$. Thus,
$\mathrm{m}^q_0(2k+2)$ is the only central extension of
$\mathrm{m}^q_0(2k+1)$.

Let us now assume that $n=2k$. If $k < q$, then $2k < 2q$ and
$2k+1 < 2q+1$. By Corollary \ref{techlemma1}, $\mathrm{g}\cong
\mathrm{m}^q_0(2k+1)$. Let now $k\ge q$. Choose the standard basis
for $\mathrm{m}^q_0(2k)$:
\[e_1, e_q,\ldots, e_{2k}\] where $[e_i, e_j] = \lambda_{ij}
e_{i+j},$ $\lambda_{1i}= 1$, $i \ge q$, and $\lambda_{ij} = 0$,
$i,j\ge q$. Let $\mathrm{g}$ be an $\mathbb{N}$-graded
one-dimensional filiform central extension of
$\mathrm{m}^q_0(2k)$. By Lemma \ref{techlemma2},   the standard
basis can be extended to the  canonical  basis $e_1, e_q,\ldots,
e_{2k+1}, e_{2k+2}$ of $\mathrm{g}$  such that where $[e_r,
e_{2k+1-r}] = \lambda_{r, 2k+1-r} e_{2k+1},$ $r=q,\ldots,k$, and
the remaining products are exactly the same as in
$\mathrm{m}^q_{0}(2k)$. Let us find unknown structure constants
$\lambda_{r, 2k+1-r}$, $r=q,\ldots, k$. Since $\mathrm{g}$ is a
Lie algebra we have that for every $r=q,\ldots, k$
\[ J(e_1, e_r, e_{2k-r}) = 0.\]
Therefore,
\begin{equation}\label{eqqq}
\lambda_{1r} \lambda_{1+r, 2k-r} + \lambda_{r, 2k-r} \lambda_{2k,
1}+ \lambda_{2k-r, 1} \lambda_{2k+1-r, r} = 0
\end{equation}
where $\lambda_{1r} = 1$, $\lambda_{2k, 1}=\lambda_{2k-r, 1} = -1$
and   $\lambda_{r, 2k-r}=0$. Equivalently,
\[ \lambda_{1+r, 2k-r} + \lambda_{r, 2k+1 - r} =0. \] Set
$\lambda_{k,k+1}=\beta$. Then
\[ \lambda_{r, 2k+1-r} = (-1)^{k-r} \beta.\]
For $\beta =0,$   $\mathrm{g}$ is isomorphic to
$\mathrm{m}^q_0(2k+1)$. Assume that $\beta\ne 0$. Then introducing
a new $\mathbb{N}$-graded basis $\{ e'_1, e'_q, \ldots,
e'_{2k+1}\}$ such that $e'_1 =e_1$, $e'_i = \beta^{-1} e_i$,
$i=q,\ldots, 2k+1$, we obtain the following structure relations
for $\mathrm{g}$:
\begin{align*}
& [e'_1, e'_i] = e'_{i+1},\,\, i=q,\ldots, 2k\\
& [e'_r, e'_{2k+1-r}] = (-1)^{r-k} e'_{2k+1}.
\end{align*}
This is a Lie algebra since $J(e_i, e_j, e_r) = 0$ for any
admissible $i<j<r$. Indeed,   if $i+j+r< 2k+1$, then $J(e_i, e_j,
e_r)=0$ since $\mathrm{m}^q_0(2k)$ is a Lie algebra. If $i+j+r =
2k+1$, then the following two cases occur.

\emph{Case 1:}  $i\ge q$. Since $q>1$ we have that $i+j< j+r< i+r
< 2k+1$. Consequently,  $\lambda_{ij}= \lambda_{jr}=\lambda_{ir} =
0$. Thus,
\[
  J(e_i, e_j, e_r)= (\lambda_{ij} \lambda_{i+j, r}+\lambda_{jr}\lambda_{j+r, i} +
\lambda_{ri}\lambda_{r+i, j}) e_{i+j+r} = 0.
\]

\emph{Case 2:}  $i=1$. Then $j+r = 2k$, $r = 2k-j$. Then
\begin{align*}
& J(e_1, e_j, e_{2k-j})=\lambda_{1j} \lambda_{1+j, 2k-j} +
\lambda_{j, 2k-j}
\lambda_{2k, 1}+ \lambda_{2k-j, 1} \lambda_{2k+1-j, j}\\
&=\lambda_{1+ j, 2k-j} + \lambda_{j, 2k+1 - j} =0
\end{align*}
The proof is complete.
\end{proof}

\begin{Definition}
The basis $e_1, \ldots, e_{2k+1}$ for $\mathrm{m}^q_{0,1}(2k+1)$
with multiplication table as in Lemma \ref{onedimextensions}  will
be called the \textbf{standard} basis.
\end{Definition}

\begin{Lemma}\label{2dextension}
Let $\mathrm{g}$ be a one-dimensional $\mathbb{N}$-graded filiform
central extension of $\mathrm{m}^q_{0,1}(2k+1)$. Then $\mathrm{g}$
is isomorphic to $\mathrm{m}^q_{0,2}(2k+2)$ defined by its basis:
$e_1, e_q, \ldots, e_{2k+1}, e_{2k+2}$ and structure relations:
\[ [e_1, e_i] = e_{i+1},\,  i=q,\ldots, 2k+1,\,\, [e_l, e_{2k+1-l}] = (-1)^{l-k} e_{2k+1},\, l=q,\ldots,
k,\]
\[ [e_r, e_{2k+2-r}] = (-1)^{r-k} (k+1-r) e_{2k+2},\, r=q,\ldots,
k+1.\]
\end{Lemma}
\begin{proof}
First of all, we determine all $\mathbb{N}$-graded one-dimensional
central extensions of $\mathrm{m}^q_{0,1}(2k+1)$ in the same way
as we did in Lemma \ref{onedimextensions}. Let $e_1, e_q, \ldots,
e_{2k+1}$ denote the standard basis for
$\mathrm{m}^q_{0,1}(2k+1)$.  Then its one-dimensional
$\mathbb{N}$-graded filiform central extension $\mathrm{g}$ can be
defined by the  canonical basis: $ e_1, e_q, \ldots, e_{2k+1},
e_{2k+2}$ (see Lemma \ref{techlemma2}). Arguing in the same way as
in Lemma \ref{onedimextensions} we obtain that $J(e_1, e_r,
e_{2k+1-r})=0$, $r=q,\ldots, k$ is equivalent to
\begin{equation}\label{eq11} \lambda_{1r}\lambda_{1+r, 2k+1-r} + \lambda_{r,
2k+1-r}\lambda_{2k+1,1} + \lambda_{2k+1-r, 1}\lambda_{2k+2-r, r} =
0.\end{equation} Note that in \eqref{eq11} $\lambda_{1r}=1,$
$\lambda_{2k+1, 1} = -1,$ $\lambda_{2k+1-r, 1} = -1$ and
$\lambda_{r, 2k+1-r} = (-1)^{r-k}$. Hence, \eqref{eq11} takes the
form
\begin{equation}
\lambda_{1+r, 2k+1-r} + \lambda_{r, 2k+2-r} = (-1)^{r-k}.
\end{equation}
This yields  $\lambda_{r, 2k+2-r} = (-1)^{r-k} (k+1 -r)$.
Therefore, $\mathrm{g}$  has the same multiplication table as
$\mathrm{m}^q_{0,2}(2k+2)$ does. We next show that
$\mathrm{m}^q_{0,2}(2k+2)$ is indeed a Lie algebra. Consider any
$i,j,r = 1, q,\ldots, 2k+2$ such that $i<j<r$ and $i+j+r=2k+2$.
The following two cases occur.

\emph{Case 1}: $i\ge q$.  Then \[i+j = 2k+2 - r < 2k+2 - q < 2k\]
since $q > 2$. Likewise, $j+r < 2k$ and $i+r < 2k$. Thus,
$\lambda_{ij} = \lambda_{jr} = \lambda_{ir} = 0$, and
\begin{align*}
& J(e_i, e_j, e_r)= (\lambda_{ij} \lambda_{i+j,
r}+\lambda_{jr}\lambda_{j+r, i} + \lambda_{ri}\lambda_{r+i, j})
e_{i+j+r} = 0.
\end{align*}

\emph{Case 2:}  $i=1$. Then $j+r = 2k+1$, $r = 2k+1-j$. Then
\begin{align*}
& J(e_1, e_j, e_{2k+1-j})=\lambda_{1j} \lambda_{1+j, 2k+1-j} +
\lambda_{j, 2k+1-j}
\lambda_{2k+1, 1}+ \lambda_{2k+1-j, 1} \lambda_{2k+2-j, j}\\
&=\lambda_{1+ j, 2k+1-j} + \lambda_{j, 2k+2 - j} - (-1)^{j-k} =0
\end{align*}
Therefore, $J(e_i, e_j, e_r)=0 $ for any $i,j,r=1, q, \ldots,
2k+2$.  This means that $\mathrm{m}^q_{0,2}(2k+2)$ is a Lie
algebra. The proof is complete.

\end{proof}

\begin{Definition} Let
$\mathrm{m}^q_{0,3}(2k+3; \beta_1)$ denote an algebra spanned by
$e_1, e_q,\ldots, e_{2k+2}, e_{2k+3}$  with the following
structure relations:
\[ [e_1, e_i] = e_{i+1},\,  i=q,\ldots, 2k+2,\,\, [e_l, e_{2k+1-l}] = (-1)^{l-k} e_{2k+1},\, l=q,\ldots,
k,\]
\[ [e_j, e_{2k+2-j}] = (-1)^{j-k} (k+1-j) e_{2k+2},\, j=q,\ldots,
k+1,\]
\[[e_r, e_{2k+3-r}] = (-1)^{r-k}\left( \binom{k-r+2}{k-r} - \beta_1\right)
e_{2k+3},\, r=q,\ldots, k+1,\] where $\beta_1$ is any scalar.
\end{Definition}

\par\medskip
\par\medskip

\begin{Definition}\label{MainConstruction}
We inductively define algebras of type $\mathrm{m}^q_{0,s}(2k+s;
\bar\beta)$ where $s\ge 3$, $\bar\beta=(\beta_1,\ldots,\beta_l)$
and $l = \left[\frac{s+1}{2}\right]-1$. An algebra of type
$\mathrm{m}^q_{0,3}(2k+3;\beta_1)$  was introduced above.  Assume
that  $\mathrm{m}^q_{0,s}(2k+s; \bar\beta)$ with a basis: $e_1,
e_q, \ldots, e_{2k+s}$ has been constructed. Then

\par\medskip
1) For an even $s$, $ \mathrm{m}^q_{0, s+1}(2k+s+1; \bar\beta')=
\text{span}\{e_1, e_q, \ldots, e_{2k+s}, e_{2k+s+1}\}$ where
$\bar\beta' = (\beta_1,\ldots,\beta_l,\beta_{l+1})$ (with
additional parameter $\beta_{l+1}$) and
\begin{equation}\label{mainformulas1}
 [e_r, e_{2k+s+1-r}] = (-1)^{k-r}\left(\binom{k-r+s}{k-r}+
\sum^{l+1}_{i=1} (-1)^i \binom{k-r+s-i}{k-r+i}\beta_i \right)
e_{2k+s+1}, \, r=q,\ldots, k+\left[ \frac{s+1}{2} \right].
\end{equation}

2) For an odd $s$, $\mathrm{m}^q_{0, s+1}(2k+s+1;
\bar\beta')=\text{span}\{e_1, e_q, \ldots, e_{2k+s}, e_{2k+s+1}\}$
where $\bar\beta' = \bar\beta= (\beta_1,\ldots,\beta_l)$  and
\begin{equation}\label{mainformulas2}
 [e_r, e_{2k+s+1-r}] = (-1)^{k-r}\left( \binom{k-r+s}{k-r}+
\sum^{l}_{i=1} (-1)^i \binom{k-r+s-i}{k-r+i}\beta_i \right)
e_{2k+s+1},\, r=q,\ldots, k+\left[ \frac{s+1}{2} \right].
\end{equation}

Additionally, $[e_1, e_{2k+s}]= e_{2k+s+1}$, and if $i+j \le
2k+s$, then  $[e_i, e_j]$ remains  the same as in
$\mathrm{m}^q_{0,s}(2k+s; \bar\beta)$.
\end{Definition}

The basis $e_1,\ldots, e_{2k+s+1}$ for $\mathrm{m}^q_{0,
s+1}(2k+s+1; \bar\beta')$ with the above multiplication table will
be called the \textbf{standard} basis.

\begin{Lemma}\label{generalext}
Let  $\mathrm{m}^q_{0,s}(2k+s; \bar\beta)$ be a Lie algebra. Then
$\mathrm{m}^q_{0,s}(2k+s; \bar\beta)$ is filiform. If $\mathrm{g}$
its one-dimensional $\mathbb{N}$-graded filiform central
extension,  then $\mathrm{g}$ is isomorphic to
$\mathrm{m}^q_{0,s+1}(2k+s+1; \bar\beta')$ for some $\bar\beta'$.
\end{Lemma}

\begin{proof} By our assumption $\mathrm{m}^q_{0,s}(2k+s;
\bar\beta)$ is a Lie algebra. As follows from Definition
\ref{MainConstruction},
\[ \mathrm{m}^q_{0,s}(2k+s; \bar\beta)= \mathrm{g}_1\oplus
\mathrm{g}_q\oplus\ldots\oplus\mathrm{g}_{2k+s}\] where
$\mathrm{g}_i = \text{span}\{ e_i\}$ is an $\mathbb{N}$-grading.
Since $[e_i, e_i] =e_{i+1},$ $i=q,\ldots,2k+s-1$, we have that
\[\mathrm{g}^2 =
\mathrm{g}_{q+1}\oplus\ldots\oplus\mathrm{g}_{2k+s},\,\,
\text{and}\,\, \mathrm{g}^i =
\mathrm{g}_{q+i-1}\oplus\ldots\oplus\mathrm{g}_{2k+s}\] where
$i>2$. Hence, $\dim\,\mathrm{g}/\mathrm{g}^2 = 2$ and
$\dim\,\mathrm{g}^i/\mathrm{g}^{i+1} = 1$ which is necessary and
sufficient condition for $\mathrm{g}$ to be filiform. It is also
easy to see that $\mathrm{m}^q_{0,s}(2k+s; \bar\beta)$ is
generated by the first two graded components.

Let us now determine all $\mathbb{N}$-graded one-dimensional
filiform central extensions of $\mathrm{m}^q_{0,s}(2k+s;
\bar\beta)$.  Let $e_1, e_q,\ldots, e_{2k+s}$ be the
\emph{standard} basis for $\mathrm{m}^q_{0,s}(2k+s; \bar\beta)$.
By Lemma \ref{techlemma2} its one-dimensional $\mathbb{N}$-graded
filiform central extension can be defined by the following
\emph{canonical} basis:
\[e_1, e_q, \ldots, e_{2k+2l-1}, e_{2k+2l}.\]  Since $\mathrm{g}$ is a
Lie algebra the Jacobian $J(e_1, e_r, e_{2k+s-r}) = 0$, $r=q,
\ldots, k+\left[\frac{s}{2}\right]$. Equivalently,
\begin{equation}\label{eq111}
\lambda_{1r}\lambda_{1+r, 2k+s-r} + \lambda_{r, 2k+s-r}
\lambda_{2k+s, 1} + \lambda_{2k+s-r, 1}\lambda_{2k+s+1-r, r} = 0
\end{equation}
where $\lambda_{1r} = 1,$  $\lambda_{2k+s, 1} =
 \lambda_{2k+s -r, 1}=-1.$
Therefore, it can be re-written as
\begin{equation}\label{eq55}
\lambda_{1+r, 2k+s-r}+\lambda_{r,2k+s+1-r}=\lambda_{r,2k+s-r}.
\end{equation}
Consider the following two cases.

\emph{Case 1}: $s=2l+1$. Then the right side of (\ref{eq55}) is

\[ \lambda_{r, 2k+s-r} = (-1)^{k-r}\left( \binom{k-r+2l}{k-r}
+  \sum^{l}_{i=1} (-1)^i \binom{k-r+2l-i}{k-r+i}\beta_i\right),\]
$r=q, \ldots, k+l$. Since $\lambda_{k+l,k+l}=0$ this system of
linear equations has a unique solution:
\begin{equation}\label{eq1111} \lambda_{r, 2k+s+1 - r} =
(-1)^{k-r} \left( \binom{k-r+2l+1}{k-r} + \sum^{l}_{i=1} (-1)^{i}
\binom{k-r+2l+1-i}{k-r+i} \beta_i\right),
\end{equation} $r=q, \ldots, k+l.$ These structure constants define $\mathrm{m}^q_{0,s+1}(2k+s+1;
\bar\beta')$ where $\bar\beta' =\bar\beta.$

\emph{Case 2}:  $s=2l$.   The right side of (\ref{eq55}):
\[ \lambda_{r, 2k+ s - r} =
(-1)^{k-r} \left( \binom{k-r+2l-1}{k-r} + \sum^{l-1}_{i=1}
(-1)^{i} \binom{k-r+2l-1-i}{k-r+i} \beta_i\right).
\]

Introducing a new parameter $\beta_l = \lambda_{k+l, k+l+1}$ we
obtain

\[ \lambda_{r, 2k+s+1 - r} =
(-1)^{k-r} \left( \binom{k-r+2l}{k-r} + \sum^{l}_{i=1} (-1)^{i}
\binom{k-r+2l-i}{k-r+i} \beta_i\right)
\]
which defines a Lie algebra of type $\mathrm{m}^q_{0, 2k+s+1}
(2k+s+1;\bar\beta') $ where $\bar\beta'=(\bar\beta, \beta_l)$.
This proves the lemma.
\end{proof}

\begin{Proposition}\label{proposition}
\label{marinaprop} {\it For any value of multiparameter $\bar\beta
= (\beta_1, \ldots,\beta_{\left[\frac{s+1}{2}\right]-1})$,
$s=1,\ldots,q,$ $\mathrm{m}^q_{0,s}(2k+s; \bar\beta)$   is a Lie
algebra. Moreover,
\par\medskip
1) if $s<q$ is odd, then $\mathrm{m}^q_{0,s}(2k+s; \bar\beta)$ has
a unique $\mathbb{N}$-graded one-dimensional central extension
which is $\mathrm{m}^q_{0,s+1}(2k+s+1;\bar\beta)$ with the same
multiparameter $\bar\beta$.
\par\medskip

2) if  $0<s< q$ is even, then $\mathrm{m}^q_{0,s}(2k+s;
\bar\beta)$ has infinitely many non-isomorphic $\mathbb{N}$-graded
one-dimensional central extensions. Each such extension is of the
form $\mathrm{m}^q_{0,s+1}(2k+s+1;\bar\beta')$ with multiparameter
$\bar\beta'=(\bar\beta,\beta_{\left[\frac{s+1}{2}\right]})$.
Moreover, for different values of
$\beta_{\left[\frac{s+1}{2}\right]}$ we obtain non-isomorphic
central extensions.}
\end{Proposition}

\begin{proof}

\par\medskip
Let us prove  by induction on $s$ that $\mathrm{m}^q_{0,s}(2k+s;
\bar\beta)$, $s=1,\dots, q$ is a Lie algebra. Lemma
\ref{2dextension} is a basis for induction when $s=1$.  Assume
that for some $s<q$ $\mathrm{m}^q_{0,s}(2k+s; \bar\beta)$ is a Lie
algebra for any $\bar\beta $. Consider $\mathrm{\bar g} =
\mathrm{m}^q_{0,s+1}(2k+s+1; \bar\beta')$. As follows from
Definition \ref{MainConstruction}, $\mathrm{\bar g}$ is obtained
from an appropriate $\mathrm{m}^q_{0,s}(2k+s; \bar\beta)$ by
extending its standard basis and adding relations
(\ref{mainformulas1}) or (\ref{mainformulas2}). Let $e_1, e_q,
\ldots, e_{2k+s}, e_{2k+s+1}$ be  the standard  basis for
$\mathrm{\bar g}$. By our inductive assumption $J(e_i, e_j,
e_k)=0$ if $i+j+r \le 2k+s$. Hence, we only need to consider the
case when $i+j+r=2k+s+1$, $i< j < r$.

If $i\ge q$, then $j+r \le 2k + (s-q)+1 \le 2k$ because $s< q$.
Since
\[ i+j< i+r< j+r \le 2k \] we have that
\[ \lambda_{i,j}=\lambda_{i,r}=\lambda_{j,r} =0. \] Therefore,
\[ J(e_i, e_j, e_k)= (\lambda_{i,j}\lambda_{i+j,k} +
\lambda_{j,k}\lambda_{j+k,i}+\lambda_{k,i}
\lambda_{k+i,j})e_{i+j+k} = 0.\]

If $i=1$, then $j+r = 2k+s$. In this case, $J(e_1, e_j,
e_{2k+s-j}) = 0$ is equivalent to \eqref{eq111}, and as was
already shown, \eqref{eq1111} is a solution to \eqref{eq111}.
Consequently, $\mathrm{m}^q_{0, 2k+s+1} (2k+s+1; \bar\beta')$ is a
Lie algebra for any values of $\bar\beta'.$

On the other hand, by Lemma \ref{generalext} any one-dimensional
filiform central extension $\mathrm{m}^q_{0,s}(2k+s; \bar\beta)$
must be of type $\mathrm{m}^q_{0,s+1}(2k+s+1; \bar\beta')$. For an
odd $s$ it is unique while for an even $s$ there is one-parameter
family of them.

\par\medskip

Let $s$ be a positive even integer such that $1< s< q$. It remains
to show that for different values of the parameter
$\beta_{\left[\frac{s+1}{2}\right]}$ we obtain non-isomorphic
$\mathbb{N}$-graded Lie algebras.  For this, we consider  two
one-dimensional central extensions of $\mathrm{m}^q_{0,s}(2k+s;
\bar\beta)$ corresponding to different values of
$\beta_{\left[\frac{s+1}{2}\right]}$:
\begin{align*}
&\mathrm{g}_1 = \text{span}\{e_1, e_q,\ldots, e_{2k+s},
e_{2k+s+1}\} \\
& [e_r, e_{2k+s+1-r}] = (f_{r,s} + (-1)^{k-r+\frac{s}{2}} \beta)
e_{2k+s+1},
\end{align*}
where
\[ f_{r,s} = (-1)^{k-r}\left( \binom{k-r+s}{k-r} +
\sum^{s/2-1}_{i=1} (-1)^i \binom{k-r+s-i}{k-r+i} \beta_i \right)
\]
 and $\beta$ is a particular value of
$\beta_{\left[\frac{s+1}{2}\right]},$ and
\begin{align*}
&\mathrm{g}_2 = \text{span}\{e'_1, e'_q, \ldots, e'_{2k+s},
e'_{2k+s+1}\} \\
& [e'_r, e'_{2k+s+1-r}] = (f_{r,s} + (-1)^{k-r+\frac{s}{2}}
\beta') e'_{2k+s+1},
\end{align*}
where $f_{r,s}$ is as above
 and $\beta'$
is another value of $\beta_{\left[\frac{s+1}{2}\right]}$ such that
$\beta\ne \beta'$. Notice that $\mathrm{g}_1$ and $\mathrm{g}_2$
have the same structure constants $\lambda_{ij}$ whenever $i+j\le
2k+s$. Let us now assume that $\mathrm{g}_1\cong \mathrm{g}_2$ as
$\mathbb{N}$-graded algebras. This means that there exists a
graded isomorphism $\varphi: \mathrm{g}_1 \to \mathrm{g}_2$
defined by
\[ \varphi(e_i) = \alpha_i e'_i \] where $i=1,q,\ldots, 2k+s+1$. Clearly, every $\alpha_i$ is a nonzero scalar.
Then $\varphi([e_1, e_i]) = [\varphi(e_1), \varphi(e_i)],$
$i=1,q,\ldots, 2k+s$. Hence, \[\alpha_{i+1} =
\alpha_1\cdot\alpha_i,\] $i=q,q+1,\ldots, 2k+s$ which means that

\begin{equation} \label{eq44} \alpha_i = \alpha^{i-q}_1\cdot\alpha_q,
\end{equation}
 where $i=q+1,
\ldots, 2k+s+1.$ Next, we can always choose $i,j>1$, $2k< i+j \le
2k+s$ such that $\lambda_{ij}\ne 0$. Then $\varphi([e_i, e_j])=
[\varphi(e_i), \varphi(e_j)],$ $\lambda_{ij}\alpha_{i+j} =
\lambda_{ij}\alpha_i\alpha_j,$ $\alpha_{i+j} =\alpha_i\alpha_j$.
Using \eqref{eq44} we get $\alpha_q = \alpha^q_1.$ It follows from
multiplication tables of $\mathrm{g}_1$ and $\mathrm{g}_2$ that
for $r_0 = k + \frac{s}{2}$ we have
\[ [e_{r_0}, e_{2k+s+1-r_0}] = \beta e_{2k+s+1},\,\,[e'_{r_0}, e'_{2k+s+1-r_0}] = \beta'
e'_{2k+s+1}.\] Therefore, $\varphi([e_{r_0}, e_{2k+s+1-r_0}]) =
[\varphi(e_{r_0}), \varphi(e_{2k+s+1-r_0})]$ which means that
\[ \beta \alpha_{2k+s+1}  = \beta' \alpha_{r_0}
\alpha_{2k+s+1-r_0}.\] Using \eqref{eq44} we obtain $\beta
\alpha^{2k+s+1}_1 = \beta' \alpha^{2k+s+1}_1$, hence, $\beta =
\beta'$ which contradicts to our original assumption. Thus,
$\varphi$ is not an isomorphism.  The proof is complete.
\end{proof}

\beR In order to simplify notation for $s$th central extension of
$\mathrm{m}^q_0(2k)$ we will omit $\bar\beta$ in $\mathrm{m}^q_{0,
2k+s}(2k+s; \bar\beta)$ whenever the value of $\bar\beta$ is not
important and denote it by  $\mathrm{m}^q_{0, 2k+s}(2k+s)$.
\eeR

We next focus on studying $m$th filiform central extensions of
$\mathrm{m}^q_0(2k)$ where $m>q$.

\begin{Lemma}\label{lemma3}
Let $k > q$, and $\mathrm{g}=\mathrm{m}^q_{0, q+s+1}(2k+q+s+1)$,
$s\ge 1$, be a Lie algebra. If $\lambda_{q, 2k+s} = 0$, then
either $\lambda_{q, 2k+s+1} = 0$ or $\lambda_{q+1, 2k+s-q} +
\lambda_{q, 2k+s-q} = 0$.
\end{Lemma}

\begin{proof}
Since $k>q$, $2k > 2q+1 = q+(q+1)$. As follows from Definition
\ref{MainConstruction} the product $[e_q, e_{q+1}]$ in
$\mathrm{g}$ must be the same as in $\mathrm{m}^q_0(2k)$.
Therefore,  $[e_q, e_{q+1}] = \lambda_{q,q+1} e_{2q+1}=0$. Since
$\mathrm{m}^q_{0, q+s+1}(2k+q+s+1)$ is a Lie algebra, we have that
$J(e_q, e_{q+1}, e_{2k+s-q}) = 0$. Equivalently,
\[ \lambda_{q+1, 2k+s-q} \lambda_{2k+s+1, q} + \lambda_{2k+s-q, q}
\lambda_{2k+s, q+1} = 0.\]

By Leibnitz rule for derivations (see the beginning of subsection
3.3)
\[ \lambda_{q+1, 2k+s} + \lambda_{q, 2k+s+1}= \lambda_{q, 2k+s}.\]
Hence, $\lambda_{q+1, 2k+s} = \lambda_{q, 2k+s} - \lambda_{q,
2k+s+1} = -\lambda_{q, 2k+s+1}.$ Thus, the above equation takes
the form:
\[ (\lambda_{q+1, 2k+s-q} + \lambda_{q, 2k+s-q}) \lambda_{q,
2k+s+1}=0. \]  Hence, either $\lambda_{q+1, 2k+s-q} + \lambda_{q,
2k+s-q}=0$ or $\lambda_{q, 2k+s+1}=0$  as required.
\end{proof}

\begin{Lemma}\label{lemma2}
Let $k>q$. If both $\mathrm{m}^q_{0, q+1}(2k+q+1)$ and
$\mathrm{m}^q_{0, q+2}(2k+q+2)$ are Lie algebras, then

(1) $\lambda_{q, 2k+1}=0,$

(2) $\lambda_{q, 2k+2}= 0$ if  $k > q+1$.
\end{Lemma}

\begin{proof}
Since $k>q$, $2k > 2q+1 = q+(q+1)$, and similarly to Lemma
\ref{lemma3} we can show that $[e_q, e_{q+1}]= \lambda_{q,q+1}
e_{2q+1} = 0$.
\par\medskip

(1) By our assumption, $\mathrm{m}^q_{0,q+1}(2k+q+1)$ is a Lie
algebra. Thus, $J(e_q, e_{q+1}, e_{2k-q}) = 0$. Equivalently,
\[ \lambda_{q+1, 2k-q} \lambda_{2k+1, q} + \lambda_{2k-q, q}
\lambda_{2k, q+1} =0 \]  As follows from multiplication table of
$\mathrm{m}^q_{0,q+1}(2k+q+1)$, $\lambda_{2k-q,q}=0$ and
$\lambda_{q+1, 2k-q} = (-1)^{q+1-k} \ne 0$. Therefore,
$\lambda_{q, 2k+1}=0$. Notice that $\lambda_{q, 2k+1} = 0$  in
$\mathrm{m}^q_{0,q+2}(2k+q+2)$ as well.
\par\medskip

(2)  Since $\lambda_{q, 2k+1}=0$ we can use Lemma \ref{lemma3} for
$s=1$. Hence, either $\lambda_{q, 2k+2} = 0$ or
\[\lambda_{q+1, 2k+1-q} + \lambda_{q, 2k+1-q}  = 0.\]
Recall that $\lambda_{q+1, 2k+1-q} = (-1)^{q+1-k}(k-q)$ and
$\lambda_{q, 2k+1-q} = (-1)^{q-k}$. Hence, if $k > q+1$, then
\[\lambda_{q+1, 2k+1-q} + \lambda_{q, 2k+1-q}  \ne 0.\]
Thus, $\lambda_{q, 2k+2} =0$, as required.

\end{proof}

\par\medskip

\subsection{Proof of Theorem \ref{MainTheorem}}

Let $\mathrm{g} = \text{span}\{e_1,\ldots,e_n\}$ be an
$\mathbb{N}$-graded Lie algebra such that

\[
[e_1,e_i] = e_{i+1},\,i=1,\ldots,n-1,\,\,[e_i,e_j]=\lambda_{ij}
e_{i+j},\,i,j>1.\]
 The Leibnitz rule for derivation
$\text{ad}(e_1)$ yields

\begin{equation}\label{LR} \lambda_{ij} = \lambda_{i+1, j} + \lambda_{i, j+1}.\end{equation}
\par\medskip
\begin{Lemma}\label{q-conditions}
Let $k> 2q$. If $\mathrm{g}=\mathrm{m}^q_{0, 2q}(2k+2q)$  is a Lie
algebra, then we have
\[ \lambda_{2q-1, 2k+1} = \lambda_{2q-2, 2k+2}= \ldots
=\lambda_{q, 2k+q} = 0.\]
\end{Lemma}

\begin{proof}
By the previous lemma we have that $\lambda_{q, 2k+1} =
\lambda_{q, 2k+2} = 0$ since $k\ne q+1$.

Since $k>2q$ we have that
\[ \lambda_{q, q+1} = \lambda_{q+1, q+2} = \ldots = \lambda_{2q-1,
2q} = 0.\] Indeed, since $i+(i+1)= 2i+1 <2k$ where $i=q,\ldots,
2q-1$ all products $[e_{i}, e_{i+1}] = \lambda_{i, i+1} e_{2i+1}$
in $\mathrm{g}$ must be the same as in $\mathrm{m}^q_0(2k)$.
Therefore, $[e_{i}, e_{i+1}] = \lambda_{i, i+1} e_{2i+1}=0$ where
$i=q,\ldots, 2q-1$.

Let us now show that $\lambda_{q+2, 2k+1}=\lambda_{q+1, 2k+2}=
\lambda_{q, 2k+3}=0$. Indeed, $J(e_{q+1}, e_{q+2}, e_{2k-q})=0$.

Since $\lambda_{q+1, q+2} = 0$ we have that

\[ \lambda_{q+2, 2k-q} \lambda_{2k+2, q+1} + \lambda_{2k-q, q+1}
\lambda_{2k+1, q+2} = 0,\]
\par\medskip
\par\medskip
\noindent where $\lambda_{q+2, 2k-q} = (-1)^{q+2-k} (k-1-q)$ and
$\lambda_{q+1, 2k-q} = (-1)^{q+1-k}.$  Hence,
\[ (k-1-q) \lambda_{q+1, 2k+2} + \lambda_{q+2, 2k+1} = 0.\]

Using relation (\ref{LR}), we have
\[ \lambda_{q+1, 2k+2} + \lambda_{q+2, 2k+1} = \lambda_{q+1,
2k+1}\] where $\lambda_{q+1, 2k+1}= \lambda_{q, 2k+1} -
\lambda_{q, 2k+2} = 0$ (by (\ref{LR}) ).

Since $k > 2q$, $k-1-q\ne 1$ we have that
\[ \lambda_{q+1, 2k+2} = \lambda_{q+2, 2k+1} =0. \]
Finally, $\lambda_{q+1, 2k+2} = \lambda_{q, 2k+2} - \lambda_{q,
2k+3}$ (by (\ref{LR}) ).  Hence, $\lambda_{q, 2k+3} = 0.$

Let us now use induction on $s$. Assume that
\[ \lambda_{q, 2k+3} =\ldots= \lambda_{q, 2k+s}=0 \] and
\[ \lambda_{q+s-1, 2k+1} = \lambda_{q+s-2, 2k+2} = \ldots =
\lambda_{q+1, 2k+s-1}=0\] where $2\le s < q.$

We know that   \[q+s-1 < q+s < 2k+2 -q -s.\] Besides,
\[(q+s-1) + (q+s) = 2q + 2s -1 \le 4q - 1 <2k. \] Thus,
$\lambda_{q+s-1, q+s} = 0.$  Therefore,  $J(e_{q+s-1}, e_{q+s},
e_{2k+2-q-s})=0$ is equivalent to
\[ \lambda_{q+s, 2k+2-q-s} \lambda_{2k+2, q+s-1} + \lambda_{2k+2-q-s, q+s-1}
\lambda_{2k+1, q+s} = 0,\] Since $\lambda_{q+s, 2k+2-q-s} =
(-1)^{q+s-k} (k+1-q-s)$ and $\lambda_{q+s-1, 2k+2-q-s} =
(-1)^{q+s-1-k}$ we have that
\[ (k+1-q-s) \lambda_{q+s-1, 2k+2} + \lambda_{q+s, 2k+1} = 0.\]
By relation (\ref{LR})

\[ \lambda_{q+s-1, 2k+2} + \lambda_{q+s, 2k+1} = \lambda_{q+s-1,
2k+1}\] where $\lambda_{q+s-1, 2k+1} = 0$ by inductive assumption.
Since $k> 2q,$ we have that $k+1-q-s\ne 1$ and $\lambda_{q+s-1,
2k+2} = \lambda_{q+s, 2k+1} = 0.$ Next $\lambda_{q+s-1, 2k+2} =
\lambda_{q+s-2, 2k+2}  - \lambda_{q+s-2, 2k+3}$. Also,
$\lambda_{q+s-2, 2k+2}=0$ by inductive assumption. Thus,
$\lambda_{q+s-2, 2k+3} = 0$.

Likewise, $\lambda_{q+s-2, 2k+3} = \lambda_{q+s-3, 2k+3} -
\lambda_{q+s-3, 2k+4}.$ Thus, $\lambda_{q+s-3, 2k+4} = 0$. After
the finite number of steps we get the following
\[ \lambda_{q+1, 2k+s} = \lambda_{q, 2k+s} - \lambda_{q,
2k+s+1}.\] Since $\lambda_{q+1, 2k+s} = \lambda_{q, 2k+s}=0$ we
have that $\lambda_{q, 2k+s+1}=0$, as required. Therefore,
\[ \lambda_{q+s, 2k+1} = \lambda_{q+s-1, 2k+2}=\ldots=\lambda_{q,
2k+s+1}=0.\]

Finally, for $s=q-1$ we obtain
\[ \lambda_{2q-1, 2k+1} = \lambda_{2q-2, 2k+2}= \ldots
=\lambda_{q, 2k+q} = 0,\] as required. The proof is complete.

\end{proof}

\begin{Lemma}\label{nocentralextensions}
Let $k> 2q$. Then $\mathrm{g}=\mathrm{m}^q_{0, 2q-1}(2k+2q-1;
\bar\beta)$ where $\bar\beta =(\beta_1,\ldots, \beta_{q-1})$  has
no one-dimensional $\mathbb{N}$-graded filiform central
extensions.
\end{Lemma}
\begin{proof}
Assume that such a central extension exists. Then it must be of
type $\mathrm{m}^q_{0, 2q} (2k+2q; \bar\beta)$ and
\[ [e_r, e_{2k+2q-r}] = (-1)^{k-r} \left( \binom{k-r+2q-1}{k-r} +
\sum^{q-1}_{i=1} (-1)^i \binom{k-r+2q-1-i}{k-r+i} \beta_i \right)
e_{2k+2q} \] where $r=q, \ldots, k+q$ and the remaining products
are the same as in $\mathrm{m}^q_{0, 2q-1}(2k+2q-1; \bar\beta)$.
Since, $k>2q$ we can apply Lemma \ref{q-conditions}. Hence, we
have that
\[\lambda_{q, 2k+q} = \lambda_{q+1, 2k+q-1}=\ldots=\lambda_{2q-2,
2k+2}=\lambda_{2q-1, 2k+1}=0. \] Equivalently, we have $q$ linear
equations:
\[ \binom{k+q-1}{k-q} + \sum^{q-1}_{i=1} (-1)^i
\binom{k+q-1-i}{k-q+i} \beta_i =0 \]
\[ \binom{k+q-2}{k-q-1} + \sum^{q-1}_{i=1} (-1)^i
\binom{k+q-2-i}{k-q-1+i} \beta_i =0 \]
\[\ldots\]
\[\binom{k}{k-2q+1} + \sum^{q-1}_{i=1} (-1)^i
\binom{k-i}{k-2q+1+i} \beta_i =0 \]

Consider the following matrix:
\par\medskip
\[ A= \left( \begin{array}{ccccc}
\binom{k}{k-1} & \binom{k+1}{k-2}& \ldots & \binom{k+q-2}{k-q+1} &
\binom{k+q-1}{k-q} \\
\binom{k-1}{k-2} & \binom{k}{k-3}& \ldots & \binom{k+q-3}{k-q} &
\binom{k+q-2}{k-q-1} \\
\ldots& \ldots&\ldots&\ldots& \ldots \\
\binom{k-q+1}{k-q} & \binom{k-q+2}{k-q-1}& \ldots &
\binom{k-1}{k-2q+2} &\binom{k}{k-2q+1}
\end{array}\right)\]
\par\medskip
Dividing each row of $A$ by its first entry (which is, of course,
nonzero) we obtain the following matrix:
\par\medskip
\[ B= \left( \begin{array}{ccccc}
1 &  f_1(x_0) & f_2(x_0) & \ldots & f_{q-1}(x_0)\\
1 &  f_1(x_1) & f_2(x_1) & \ldots & f_{q-1}(x_1)\\
\ldots & \ldots& \ldots & \ldots &\ldots \\
1 &  f_1(x_{q-1})& f_2(x_{q-1})&\ldots &
f_{q-1}(x_{q-1})\end{array} \right)
\] where
\[ f_1(x)= x(x+2),\,
f_2(x)=(x-1)x(x+2)(x+3),\ldots,f_{q-1}(x)=(x-q+2)\cdots
x(x+2)\cdots (x+q)\] and
\[ x_0 = k-1, \, x_1=k-2,\ldots, x_{q-1}= k-q.\]
We next want to prove that $B$ is a nonsingular matrix. Let
\[ y=f_1(x),\,g_2(y)= y(y-3),\ldots, g_{q-1}(y)=y(y-3)\cdots(y+2q-q^2).\] Then
\[ B= \left( \begin{array}{ccccc}
1 &  y_0 & g_2(y_0) & \ldots & g_{q-1}(y_0)\\
1 &  y_1 & g_2(y_1) & \ldots & g_{q-1}(y_1)\\
\ldots & \ldots& \ldots & \ldots &\ldots \\
1 &  y_{q-1}& g_2(y_{q-1})&\ldots & g_{q-1}(y_{q-1})\end{array}
\right)
\] where
\[y_0=x^2_0+2x_0,\,
y_1=x^2_1+2x_1,\ldots,y_{q-1}=x^2_{q-1}+2x_{q-1}.\] Note that
$\text{deg}\, g_i = i,$ $i=2,\ldots, q-1$. According to the
statement on p 319 (see \cite{shalev:97}) $B$ is a nonsingular
matrix. Therefore, $A$ is also non-singular. This implies that the
above system of linear equations is \emph{inconsistent}. Hence,
$\mathrm{m}^q_{0, 2q}(2k+2q)$ is not a Lie algebra. This implies
that $\mathrm{g}$ has no one-dimensional $\mathbb{N}$-graded
filiform central extensions. The proof is complete.

\end{proof}

Let us now finish the proof of Theorem \ref{MainTheorem}. Consider
$\mathrm{g} = \bigoplus^{\infty}_{i=1,q} \mathrm{g}_i$ generated
by both $\mathrm{g}_1$ and $\mathrm{g}_q$ that satisfies
\begin{equation}\label{maincondition}
 [\mathrm{g}_q, \mathrm{g}_{q+1}] = [\mathrm{g}_{q+1},
\mathrm{g}_{q+2}]=\ldots=[\mathrm{g}_{2q}, \mathrm{g}_{2q+1}]=0
\end{equation}
By Corollary \ref{corollary1}, we can choose a basis for
$\mathrm{g}:$ $\{e_1, e_q, e_{q+1},\ldots\}$ where $[e_1, e_i]
=e_{i+1}$, and $\mathrm{g}_i = \text{span}\{e_i\}$, $i>1$. Hence,
condition (\ref{maincondition}) is equivalent to

\begin{equation}
\lambda_{q, q+1}= \lambda_{q+1, q+2} = \ldots = \lambda_{2q,
2q+1}=0
\end{equation}
where $\lambda_{ij}$ are corresponding structure constants. Set
$r=2q$. By Corollary \ref{techlemma1} $\mathrm{g}(r)$ is
isomorphic to $\mathrm{m}^q_0(r)$.    By Lemma
\ref{onedimextensions} and condition (\ref{maincondition}) we have
that
\begin{align*}
& \mathrm{g}(r+1) \cong \mathrm{m}^q_0(r+1),\\
& \mathrm{g}(r+2) \cong \mathrm{m}^q_0(r+2),\\
& \ldots\\
& \mathrm{g}(r+2q+2)\cong \mathrm{m}^q_0(r+2q+2).
\end{align*}
Note that $\mathrm{m}^q_0(r+2q+2) = \mathrm{m}^q_0(4q+2) =
\mathrm{m}^q_0(2(2q+1)) =\mathrm{m}^q_0(2k)$ where $k=2q+1$. Thus,
$\mathrm{g}$ is obtained by taking one-dimensional central
extensions of $\mathrm{m}^q_0(2k)$ where $k > 2q$. If $\mathrm{g}$
is not of type $\mathrm{m}^q_0$, then by taking one-dimensional
central extensions at some point we obtain
$\mathrm{m}^q_{0,1}(2l+1)$ where $l>2q$. Then  the  $(2q-1)$st
filiform central extension will be of type
$\mathrm{m}^q_{0,2q-1}(2l+2q-1)$, $l>2q,$ and by Lemma
\ref{nocentralextensions} it has no required central extensions, a
contradiction. Therefore, $\mathrm{g}$ must be of type
$\mathrm{m}^q_0$. The proof is complete.
\par\medskip
\par\medskip
\subsection{The case of $q=3$}

The purpose of this section is to prove the following theorem:

\begin{Theorem}\label{caseq3}
Let  $\mathrm{g}= \bigoplus_{i\in \mathbb{N}}\mathrm{g}_i$ be an
infinite-dimensional $\mathbb{N}$-graded  Lie algebra of maximal
class and suppose $\mathrm{g} = \langle \mathrm{g}_1, \mathrm{g}_3
\rangle $. Then one of the following holds

(1) $\mathrm{g}\cong \mathrm{m}^3_0$;

(2) $\mathrm{g}\cong \mathrm{m}_3$;

(3) $\mathrm{g}\cong W^3$.
\end{Theorem}

The proof of this theorem will follow from the series of lemmas
below.

Let $\mathrm{g}$ be an $\mathbb{N}$-graded Lie algebra of maximal
class generated by both $\mathrm{g}_1$ and $\mathrm{g}_3$ (not
generated by $\mathrm{g}_1$ or $\mathrm{g}_3$ only). Hence, it has
the following $\mathbb{N}$-grading: $\mathrm{g}=
\bigoplus^{\infty}_{i=1,3} \mathrm{g}_i$. At the beginning of
section 3 we introduced Lie  algebras of types $\mathrm{m}^q_0$,
$\mathrm{m}_q$ and $W^q$. For $q=3$ we will show that these are
the only $\mathbb{N}$-graded Lie algebras of maximal class
generated by $\mathrm{g}_1$ and $\mathrm{g}_3$.

Recall that, by definition,  $\mathrm{m}_3(l)=\text{span}\{e_1,
e_3, \ldots, e_l\}$ such that $[e_1, e_i] = e_{1+i}$, $i=3,\ldots,
l-1$ and $[e_3, e_j]=e_{3+j}$, $j=4,\ldots, l-3.$ We can assume
that $l\ge 7$ and write $l= 6+n$ where $n\ge 1$.

\begin{Lemma}
Let $\mathrm{g}$ be a filiform Lie algebra of type
$\mathrm{m}_3(6+n)$ where $n\ge 1$. Then $\mathrm{g}$ is
isomorphic to $\mathrm{m}^3_{0,n}(6+n; \bar\beta)$ ($n$th
filiform central extension of $\mathrm{m}^3_0(6)$ ) where
$\bar\beta = (0,\ldots, 0)$.
\end{Lemma}
\begin{proof}
Let us compare multiplication tables of both algebras. First,
$\mathrm{m}^3_{0, n}(6+n;\bar\beta) = \text{span}\{e_1, e_3,
\ldots, e_{6+n}\}$ where $\bar\beta=(0,\ldots,0)$ has the
following multiplication table: $[e_r, e_{6+l-r}] = (-1)^{3-r}
\binom{3-r+l-1}{3-r} e_{6+l}$, $r=3,\ldots, 3+l/2$,
$l=1,\ldots,n$. If $r=3$,  then the above binomial coefficient is
1. If $r>3$, then it is zero.  Therefore, the multiplication table
of $\mathrm{m}^3_{0, n}(6+n;\bar\beta)$ for $\bar\beta =(0,\ldots,
0)$ is exactly the same as that of $\mathrm{m}_3(6+n)$. Hence,
they are isomorphic.
\end{proof}

\begin{Lemma}
If $n=2l+1$, then  $\mathrm{m}_3(6+2l+1)$ has a unique
one-dimensional $\mathbb{N}$-graded filiform central extension
which is $\mathrm{m}_3(6+2l+2)$. If $n=2l$, then
$\mathrm{m}_3(6+2l)$ has a one-parameter family of
$\mathbb{N}$-graded filiform central extensions of type
$\mathrm{m}^3_{0,2l+1}(6+2l+1; \bar\beta')$ where
$\bar\beta'=(0,\ldots,0,\beta_l)$, $\beta_l$ is a scalar.

\end{Lemma}

\begin{proof}

If $n=2l+1$, then $\mathrm{m}_3(6+2l+1)$ is isomorphic to
$\mathrm{m}^3_{0,2l+1}(6+2l+1; \bar\beta)$,
$\bar\beta=(0,\ldots,0)$. By Lemma \ref{generalext}, it has a
unique one-dimensional $\mathbb{N}$-graded filiform central
extension which is $\mathrm{m}^3_{0,2l+2}(6+2l+2; \bar\beta)$,
$\bar\beta=(0,\ldots,0)$ and, therefore, is isomorphic to
$\mathrm{m}_3(6+2l+2)$. If $n=2l$, then $\mathrm{m}_3(6+2l)$ is
isomorphic to $\mathrm{m}^3_{0,2l}(6+2l; \bar\beta)$,
$\bar\beta=(0,\ldots,0)$. By Lemma \ref{generalext}, it has a
one-parameter family of the required central extensions
$\mathrm{m}^3_{0,2l+1}(6+2l+1; \bar\beta')$,
$\bar\beta'=(0,\ldots,0, \beta_l)$.
\end{proof}

\begin{Lemma}
Let $\mathrm{g} = \mathrm{m}^3_{0,2l+1}(6+2l+1; \bar\beta')$ where
$l\ge 3$, $\bar\beta'=(0,\ldots,0,\beta_l)$, $\beta_l\ne 0$, be a
Lie algebra. Then it has no one-dimensional $\mathbb{N}$-graded
filiform central extensions.
\end{Lemma}
\begin{proof}
Assume the contrary, that is, a one-dimensional
$\mathbb{N}$-graded filiform central extension of $\mathrm{g}$
exists. Then by Lemma \ref{generalext} it must be
$\mathrm{m}^3_{0,2l+2}(6+2l+2; \bar\beta')$ where
$\bar\beta'=(0,\ldots,0,\beta_l)$, and
\[  [e_r, e_{6+2l+2-r}] = (-1)^{3-r} \left(  \binom{3-r+2l+1}{3-r} + (-1)^l \binom{3-r+2l+1-l}{3-r+l}\beta_l \right)e_{8+2l}.
\] The corresponding structure constants are $\lambda_{r, 8+2l-r} =
(-1)^{3-r} \left( \binom{3-r+2l+1}{3-r} + (-1)^l (3-r+l+1) \beta_l
\right).$ This means that if $r=3$, then $\lambda_{3, 5+2l} = 1 +
(-1)^l (l+1) \beta_l$. If $r > 3$, then $\lambda_{r, 8+2l-r} =
(-1)^{3-r+l} (3-r+l+1) \beta_l$. Since
$\mathrm{m}^3_{0,2l+2}(8+2l; \bar\beta')$ is a Lie algebra,
$J(e_3, e_j, e_k)=0$ where $3+j+k=8+2l$ and $3<j<k$. In terms of
structure constants this can be re-written as:
\[ \lambda_{3,j}\lambda_{3+j, k} + \lambda_{j,k} \lambda_{5+2l,3}
+ \lambda_{k,3} \lambda_{k+3,j} = 0.\] Note that
\[ 3+j < 3+k < j+k =5+2l < 6+2l. \] Hence, the Lie products $[e_3, e_j]$, $[e_3, e_k]$ and
$[e_j, e_k]$ are the same as in $\mathrm{m}^3_{0, 2l}(6+2l;
\bar\beta)$ where $\bar\beta = (0,\ldots,0)$, which is isomorphic
to $\mathrm{m}_3(6+2l)$. Thus, $\lambda_{3,j} = \lambda_{3, k} =1$
and $\lambda_{j, k} = 0$ since both $j,k >3$. The above equation
takes the form $\lambda_{k, 3+j} = \lambda_{j, k+3}$. Since
$\beta_l\ne 0$, $(-1)^{3-k+l}(3-k+l+1) = (-1)^{3-j+l}(3-j+l+1)$.
Since $j+k= 5+2l$, an odd number, one of $j$, $k$  is even while
the other one is odd. Therefore, $ (3-k+l+1) = -(3-j+l+1)$, $k+j =
8+2l$, a contradiction. The lemma is proved.
\end{proof}

\beR\label{Mremark} It follows from the above lemmas that
$\mathrm{m}_3$ is the only infinite-dimensional Lie algebra that
can be obtained by taking one-dimensional $\mathbb{N}$-graded
filiform central extensions of $\mathrm{m}_3(n)$, $n\ge 12$. \eeR

Next we will be interested in one-dimensional $\mathbb{N}$-graded
filiform central extensions of $W^3(n)$. Recall that $W^q(n)$ is
defined by its basis $\{e_1, e_q,\ldots, e_n\}$ and relations:
$[e_i, e_j]=(j-i)e_{i+j}$ whenever $i+j\le n$, and the products
equal to 0, otherwise. By setting $y_1=e_1$ and $y_i = 60(i-2)!
e_i$ $q\le i\le n$, we have that $[y_1, y_i]= y_{i+1}$,
$i=q,\ldots,n-1,$ and $[y_i, y_j] = \lambda_{i,j} y_{i+j}$,
$i+j\le n$ where
\begin{equation}\label{Wstructureconstants} \lambda_{i,j}= \frac{60(i-2)!
(j-2)!(j-i)}{(i+j-2)!}.
\end{equation}

\begin{Lemma}\label{Wextensions}
A Lie algebra of type $W^3(n)$, $n\ge 14$, has a unique
one-dimensional $\mathbb{N}$-graded filiform central extension
which is isomorphic to $W^3(n+1)$.
\end{Lemma}

\begin{proof}
First of all, we note that $W^3(n+1)$ is a required central
extension of $W^3(n)$. We only need to show that it is unique.
Choose a basis for $W^3(n)$: $\{ e_1, e_3, \ldots, e_n \}$ with
$[e_1, e_i] =e_{i+1}$, $i=1,\ldots, n-1$, and $[e_i, e_j]
=\lambda_{i,j} e_{i+j}$, $i+j\le n$.  If we set $I =
\text{span}\{e_7, \ldots e_n\}$, then $I$ is an ideal of $W^3(n)$
and $W^3(n)/I$ is isomorphic to $\mathrm{m}^3_0(6)$. Therefore,
$W^3(n)$ is obtained from $\mathrm{m}^3_0(6)$ by taking
one-dimensional $\mathbb{N}$-graded filiform central extensions.
Let $n= 6 + r$ where $r\ge 1$.  Then $W^3(6+r)\cong
\mathrm{m}^3_{0,r}(6+r; \bar\beta)$ where $\bar\beta=(\beta_1,
\ldots, \beta_l)$, $l=\left[\frac{r+1}{2}\right] -1$. If $r$ is
odd, then by Lemma \ref{generalext} it has a unique
one-dimensional $\mathbb{N}$-graded filiform central extension. If
$r$ is even,  then by Lemma \ref{generalext} a one-dimensional
$\mathbb{N}$-graded filiform central extension  of
$\mathrm{m}^3_{0,r}(6+r; \bar\beta)$ is
$\mathrm{m}^3_{0,r+1}(6+r+1; \bar\beta')$ where
$\bar\beta'=(\beta_1, \ldots, \beta_l, \beta_{l+1})$. Set
$\beta_{l+1}=t$. To show the uniqueness we express $t$  in terms
of structure constants of $\mathrm{m}^3_{0,r}(6+r; \bar\beta)$
that are all known. As follows from (\ref{mainformulas1}), we have
that
\[ [e_s, e_{7+r-s}] = (-1)^{3-s}\binom{3-s+r}{3-s} +
\sum^{l}_{i=1} (-1)^{3+i-s}\binom{3-s+r-i}{3-s+i} \beta_i +
(-1)^{3-s+l} t. \] Hence,
\begin{equation}\label{structureconstants}
\lambda_{s,7+r-s} = (-1)^{3-s+l} t+A_s \end{equation} where $A_s$
depends on parameters of $\mathrm{m}^3_{0,r}(6+r; \bar\beta)$,
$s=3,\ldots, r/3+3.$ Since $\mathrm{m}^3_{0,r+1}(6+r+1;
\bar\beta')$ must be  a Lie algebra, the Jacobi identity $J(e_3,
e_4, e_{r})=0$ yields
\begin{equation}\label{lambdaeq}
\lambda_{3,4} \lambda_{7, r} + \lambda_{4, r} \lambda_{4+r, 3} +
\lambda_{r,3} \lambda_{r+3, 4} =0
\end{equation}
where $\lambda_{3,4}=1$, and $\lambda_{4,r}$, $\lambda_{r,3}$ are
known since $3+r<4+r<7+r$. Applying (\ref{structureconstants}),
$\lambda_{7, r} = (-1)^{l-4}t + A_7$, $\lambda_{3, 4+r} = (-1)^l t
+A_3$ and $\lambda_{4, r+3} = (-1)^{l-1} t + A_4$. Therefore,
(\ref{lambdaeq}) gives rise to a linear equation in $t$.  Applying
(\ref{Wstructureconstants}),  the coefficient of $t$ in this
equation is $1+ \lambda_{r,3} - \lambda_{4,r}= 1 +
60\frac{14-r-r^2}{(r-1)r(r+1)(r+2)}.$ When $r\ge 8$ it is nonzero.
Thus, $t$ is uniquely determined from (\ref{lambdaeq}). This
completes the proof.
\end{proof}

\beR\label{Wremark} It follows from the above lemmas that $W^3$ is
the only infinite-dimensional Lie algebra that can be obtained by
taking one-dimensional $\mathbb{N}$-graded filiform central
extensions of ${W}^3(n)$, $n\ge 14$. \eeR

\beR\label{kmorethan6} As follows from Lemma
\ref{nocentralextensions},  $\mathrm{m}^3_0$ is the only
infinite-dimensional Lie algebra that can be obtained by
considering one-dimensional $\mathbb{N}$-graded filiform central
extensions of $\mathrm{m}^3_0(2k)$ where $k>6$. \eeR

Further we have to consider the remaining cases: $k=3, 4, 5$ or 6.
We first deal with cases when $k=4, 5$ or 6.

\begin{Lemma}\label{kequals4}
Let $k=4$. Then $\mathrm{m}^3_0$ is the only infinite-dimensional
Lie algebra obtained by taking one-dimensional $\mathbb{N}$-graded
filiform central extensions of $\mathrm{m}^3_0(8)$.
\end{Lemma}
\begin{proof}
Consider $\mathrm{g}=\mathrm{m}^3_0(8)$. By Lemma
\ref{onedimextensions} its one-dimensional $\mathbb{N}$-graded
filiform central extension is either $\mathrm{m}^3_0(9)$ or
$\mathrm{m}^3_{0, 1}(9)$. If it is $\mathrm{m}^3_0(9)$, then
applying the same lemma again, its one-dimensional
$\mathbb{N}$-graded filiform central extension must be of type
$\mathrm{m}^3_0(10)= \mathrm{m}^3_0(2\cdot5)$ which leads to  the
case $k=5$ that will be covered in the next lemma. Without loss of
generality we assume that the second possibility holds. As follows
from Lemma \ref{lemma2} (1), $\lambda_{3,9} =0$ in
$\mathrm{m}^3_{0, 4}(12)$. Using formulas (\ref{mainformulas1}),
(\ref{mainformulas2}), we obtain
\[ \lambda_{3,9} = (-1)^{3-4}\left( \binom{4-3+3}{4-3} -
\binom{4-3+2}{4-3+1} \beta_1 \right) = 0\] Hence, $\beta_1 =
\frac{4}{3}$.  Next the Jacobi identity $J(e_3, e_4, e_7)=0$  in
$\mathrm{m}^3_{0,6}(14)$ is equivalent to
\begin{equation}\label{eq1} \lambda_{3,4}\lambda_{7,7} +
\lambda_{4,7}\lambda_{11,3} +\lambda_{7,3}\lambda_{10,4} =
0\end{equation} where $\lambda_{7,7} = 0$, $\lambda_{4,7} =
(-1)^{4-4}\left( \binom{4-4+2}{4-4} - \beta_1 \right) =
-\frac{1}{3}$, $\lambda_{3,7} = -2$, $\lambda_{3,11} =
-\binom{6}{1} + \binom{5}{2} \beta_1 - \binom{4}{3} \beta_2 = -6
+10 \beta_1 - 4\beta_2,$ $\lambda_{4,10} = \binom{5}{0} -
\binom{4}{1} \beta_1 + \binom{3}{2}\beta_2 = 1-4\beta_1
+3\beta_2$. Hence, (\ref{eq1}) gives rise to
$-\frac{1}{3}(6-10\beta_1 + 4\beta_2)+ 2(-1+4\beta_1 -3\beta_2)
=0$. Thus, $\beta_2 = \frac{50}{33}$. In $\mathrm{m}^3_{0,7}(15)$
we consider the Jacobi identity $J(e_3, e_4, e_8) = 0$ which is
equivalent to
\begin{equation}\label{eq2}
 -\lambda_{4,8} \lambda_{3,12} +
\lambda_{3,8} \lambda_{4,11} = 0
\end{equation}
where $\lambda_{3,8} = -\frac{5}{3}$, $\lambda_{4,8} =
-\frac{5}{3}$, $\lambda_{4,11}=1-5\beta_1+6\beta_2-\beta_3$,
$\lambda_{3,12} = -7+15\beta_1 - 10\beta_2+\beta_3$. Then
(\ref{eq2}) implies that $\beta_3 = \frac{92}{33}$. In
$\mathrm{m}^3_{0,8}(16)$ we consider the Jacobi identity $J(e_3,
e_5, e_8) = 0$ which is equivalent to
\begin{equation}\label{eq22}
 -\lambda_{5,8} \lambda_{3,13} +
\lambda_{3,8} \lambda_{5,11} = 0
\end{equation}
where $\lambda_{5,8} = -\frac{2}{11}, $ $\lambda_{3,13} =
-8+21\beta_1-20\beta_2+5\beta_3 = \frac{40}{11}$, $\lambda_{5,11}
=\frac{40}{11}$. Substituting these values into (\ref{eq22}) we
obtain that the left side of it is nonzero, a contradiction.

The proof is complete.
\end{proof}

\begin{Lemma}\label{kequals5}
Let $k=5$. Then $\mathrm{m}^3_0$ is the only infinite-dimensional
Lie algebra obtained by taking one-dimensional $\mathbb{N}$-graded
filiform central extensions of $\mathrm{m}^3_0(10)$.
\end{Lemma}
\begin{proof}
Consider $\mathrm{g}=\mathrm{m}^3_0(10)$. By Lemma
\ref{onedimextensions} its one-dimensional $\mathbb{N}$-graded
filiform central extension is either $\mathrm{m}^3_0(11)$ or
$\mathrm{m}^3_{0, 1}(11)$. If it is $\mathrm{m}^3_0(11)$, then
applying the same lemma again, its one-dimensional
$\mathbb{N}$-graded filiform central extension must be of type
$\mathrm{m}^3_0(12)= \mathrm{m}^3_0(2\cdot 6)$ which leads to  the
case $k=6$  covered in the next lemma. Without loss of generality
we assume that the second possibility holds. As follows from Lemma
\ref{lemma2}, $\lambda_{3,11} =0$ in $\mathrm{m}^3_{0, 4}(14)$ and
$\lambda_{3,12}=0$ in $\mathrm{m}^3_{0, 5}(15)$. Using formulas
(\ref{mainformulas1}), (\ref{mainformulas2}), we obtain
$\lambda_{3,11} = (-1)^{3-5} \left( \binom{5-3+3}{5-3} -
\binom{5-3+2}{5-3+1} \beta_1\right) = 10 - 4\beta_1 =0$. Hence,
$\beta_1 = \frac{5}{2}.$ Next, $\lambda_{3,12} = (-1)^{3-5} \left(
\binom{5-3+4}{5-3} - \binom{5-3+3}{5-3+1} \beta_1 +\beta_2 \right)
= 15 - 10\beta_1+\beta_2 =0$. Hence, $\beta_2 = 10.$ The Jacobi
identity $J(e_3, e_5, e_8)=0$ implies
\begin{equation}\label{eq3}
\lambda_{3,5}\lambda_{8,8} + \lambda_{5,8} \lambda_{13, 3} +
\lambda_{8,3} \lambda_{11,5} =0
\end{equation}
where $\lambda_{8,8}=0$, $\lambda_{5,8} = (-1)^{5-5} \left(
\binom{0+2}{0} - \beta_1 \right) = -\frac{3}{2},$ $\lambda_{3,8}
=1$, $\lambda_{3,13} = \binom{7}{2} - \binom{6}{3}\beta_1 +
\binom{5}{4} \beta_2 = 21-20\beta_1+5\beta_2 = 21$, and
$\lambda_{5,11} = (-1)^0 \left( \binom{0+5}{0}-\binom{0+4}{1}
\beta_1 +\binom{3}{2} \beta_2\right) = 1-4\beta_1+3\beta_2= 21.$
However, the left side of (\ref{eq3}) does not equal to 0, a
contradiction. This completes the proof.
\end{proof}

\begin{Lemma}\label{kequals6}
Let $k=6$. Then  $\mathrm{m}^3_0$ is the only infinite-dimensional
Lie algebra obtained by taking one-dimensional $\mathbb{N}$-graded
filiform central extensions of $\mathrm{m}^3_0(12)$.
\end{Lemma}

\begin{proof}
Consider $\mathrm{g}=\mathrm{m}^3_0(12)$. By Lemma
\ref{onedimextensions} its one-dimensional $\mathbb{N}$-graded
filiform central extension is either $\mathrm{m}^3_0(13)$ or
$\mathrm{m}^3_{0, 1}(13)$. If it is $\mathrm{m}^3_0(13)$, then
applying the same lemma again, its one-dimensional
$\mathbb{N}$-graded filiform central extension must be of type
$\mathrm{m}^3_0(14)= \mathrm{m}^3_0(2\cdot 7)$ which leads to  the
case $k>6$  covered in the Remark \ref{kmorethan6}. Without loss
of generality we assume that the second possibility holds. As
follows from Lemma \ref{lemma2}, $\lambda_{3,13} =0$ in
$\mathrm{m}^3_{0, 4}(16)$ and $\lambda_{3,14}=0$ in
$\mathrm{m}^3_{0, 5}(17)$. It follows from the proof of Lemma
\ref{q-conditions} that in this case $\lambda_{3, 2\cdot 6 + 3} =
\lambda_{3, 15}$ is also zero.
 Using formulas (\ref{mainformulas1}),
(\ref{mainformulas2}), we obtain
\[ \lambda_{3,13} = (-1)^{3-6} \left( \binom{6-3+3}{6-3} -
\binom{6-3+2}{6-3+1} \beta_1 \right) =
-\binom{6}{3}+\binom{5}{4}\beta_1 = 0,\]
\[\lambda_{3,14} = (-1)^{3-6} \left( \binom{6-3+4}{6-3} -
\binom{6-3+3}{6-3+1} \beta_1 +\beta_2\right) =
-\binom{7}{3}+\binom{6}{4}\beta_1 -\beta_2 = 0,\]

\[\lambda_{3,15} = (-1)^{3-6} \left( \binom{6-3+5}{6-3} -
\binom{6-3+4}{6-3+1} \beta_1 +\binom{6-3+3}{6-3+2}\beta_2\right) =
-\binom{8}{3}+\binom{7}{4}\beta_1 - 6\beta_2 = 0.\] However, this
system of 3 linear equations in 2 variables is inconsistent. The
proof is complete.

\end{proof}

Finally we study one-dimensional $\mathbb{N}$-graded filiform
central extensions of $\mathrm{m}^3_0(6)$. By Lemma
\ref{onedimextensions} its one-dimensional $\mathbb{N}$-graded
filiform central extension is either $\mathrm{m}^3_0(7)$ or
$\mathrm{m}^3_{0, 1}(7)$. If it is $\mathrm{m}^3_0(7)$, then
applying the same lemma again, its one-dimensional
$\mathbb{N}$-graded filiform central extension must be of type
$\mathrm{m}^3_0(8)= \mathrm{m}^3_0(2\cdot 4)$ which leads to  the
case $k=4$  covered in Lemma \ref{kequals4}. Without loss of
generality we assume that the second possibility holds.
 Taking one-dimensional filiform central extensions of
 $\mathrm{m}^3_0(6)$, in ten steps we obtain $\mathrm{m}^3_{0,
 10}(16; \beta_1, \beta_2, \beta_3, \beta_4)$. Since it must be a
 Lie algebra  we have that $J(e_3, e_4, e_5)= 0$, $J(e_3, e_5,
 e_6)=0$ and also $J(e_3, e_4, e_8)=0$. Re-writing these
 identities in terms of structure constants, and then using
 formulas (\ref{mainformulas1}), (\ref{mainformulas2}) to express
 each structure constant in terms of $\beta_1, \beta_2, \beta_3$
 and $\beta_4$ we obtain the following relations:
 \begin{equation}\label{betarelations}
\beta_2 =\frac{4\beta^2_1}{3(1+\beta_1)},\,\,
\beta_3=\frac{5\beta_1\beta_2-10\beta^2_2}{3-4\beta_2-2\beta_1},\,\,
\beta_4=\frac{-5\beta^2_2 +
6\beta_2\beta_3+4\beta_1\beta_3}{2\beta_1+\beta_2}.
 \end{equation}
Considering the Jacobi identity $J(e_3, e_4, e_9)=0$  we obtain
\begin{equation}\label{eq4}
\beta_4 +(\beta_3 -3\beta_2
+\beta_1)(-1+8\beta_1-21\beta_2+20\beta_3-5\beta_4)+(3\beta_2-4\beta_1+1)(\beta_1-6\beta_2+10\beta_3-4\beta_4)=0.
\end{equation}
Using (\ref{betarelations}), the equation (\ref{eq4}) can be
re-written as the following equation in one variable $\beta_1$:
\begin{equation}\label{eq5}
245\beta^{10}_1 + 238 \beta^9_1 - 606\beta^8_1 +270 \beta^7_1 -27
\beta^6_1 =0
\end{equation}
that has the following roots: $\beta_1=0, \frac{3}{5},
\frac{1}{7},$ and $\frac{-6\pm 3\sqrt{11}}{7}.$ If $\beta_1=
\frac{3}{5}$ or $\frac{-6\pm 3\sqrt{11}}{7}$ then it is easy to
check that $J(e_3, e_5, e_8)\ne 0$. If $\beta_1=0$, then
$\mathrm{m}^3_{0, 9}(15; \beta_1, \beta_2, \beta_3, \beta_4)$ is
isomorphic to $\mathrm{m}_3(15)$, and by Remark \ref{Mremark}
leads to infinite-dimensional Lie algebra $\mathrm{m}_3$. If
$\beta_1= \frac{1}{7}$, then  $\mathrm{m}^3_{0, 9}(15; \beta_1,
\beta_2, \beta_3, \beta_4)$ is isomorphic to $W^3(15)$, and  by
Lemma \ref{Wextensions}, it leads to $W^3$.

\appendix{}

\section{}
\label{app1}

Here we recall some facts about topology of varieties which are
relevant to proofs in Section \ref{topology}.

\subsection{The case of smooth varieties}
Let $X_d\subset\Pr^n$ be a smooth complex projective hypersurface
of degree $d$, i.e. the set of points satisfying
$\{f_d(x_0,\dots,x_n)=0\}$, where $f_d$ is a homogeneous
polynomial of degree $d$, defining $X_d$. (This in particular
implies that $f_d$ is an irreducible polynomial, $X_d$ is an
irreducible algebraic variety.) Any two such hypersurfaces (of the
same degree) are diffeomorphic as smooth manifolds. In particular,
their topological invariants (and some of their geometric
invariants) are determined by the pair $(d,n)$.
\bprop \cite[Ch. 5]{Dimca:92}
The integral homologies are torsion free and satisfy:
$H_{2i+1}(X_d,\Z)=0$ for $2i+1\neq n-1$, $H_{2i}(X_d,\Z)=\Z$ for
$2n-2\ge 2i\ge0$ and $2i\neq n-1$  and
 $b_{n-1}(X_d)=\frac{(1-d)^{n+1}-1}{d}+\Big\{\begin{array}{l} -1:\ n-even\\2:\ n-odd\end{array}$.
\\The topological Euler characteristic is $\chi(X_d)=\frac{(1-d)^{n+1}-1}{d}+n+1$.
\eprop

\bex For $n=2$ we get a smooth plane curve of (topological)
genus $g=\binom{d-1}{2}$ and with topological Euler characteristic
$\chi(X_d)=2-2g$. \eex

Let $\{X_i\}_{i=1,\dots,k}\subset\Pr^n$ be some  hypersurfaces.
Assume $X:=\capl^k_{i=1}X_i$ is a complete intersection.
Let $d_j$ denote the degree of $X_j$.

As in the case of hypersurfaces, any two {\em smooth} complete
intersections with the same multi-degree $(d_1,\dots,d_k)$ are diffeomorphic. So,
the topological invariants (and some geometric invariants too) are
completely determined by the numbers $(d_1,\dots,d_k,n)$.
\bprop
Let $X_{\ud}\subset\Pr^n$ be a smooth complete intersection of
multidegree $\ud:=(d_1,\dots,d_k)$.
The integral homologies are torsion free and satisfy:
$H_{2i+1}(X_\ud,\Z)=0$ for $2i+1\neq n-k$, $H_{2i}(X_\ud,\Z)=\Z$
for $2n-2k\ge 2i\ge0$ and $2i\neq n-k$  and
 $rank\Big(H_{n-k}(X_\ud,\Z)\Big)=\chi(X_\ud)-(n-k)$ for $(n-k)$ even,
  $rank\Big(H_{n-k}(X_\ud,\Z)\Big)=n-k+1-\chi(X_\ud)$ for $(n-k)$ odd.
 Here the topological Euler characteristic, $\chi(X_\ud)=(\prodl^k_{i=1}d_i)Coeff_{x^{n-k}}\frac{(1+x)^{n+1}}{\prodl^k_{i=1}(1+d_ix)}$.

\eprop
References for this are \cite[Ch. 5]{Dimca:92} and \cite[Appendix 1]{hirzebruch:66}.
The proof is based on the Lefschetz
hyperplane section theorem.
The Euler characteristic can be obtained e.g. as the top Chern
class $c_n(T_{X_\ud})$ from the exact sequence $0\to T_{X_\ud}\to
T_{\P^n}|_{X_\ud}\to\cN_{{X_\ud}/\P^n}\to0$ and the fact that
$\cN_{{X_\ud}/\P^n}=\oplus\cO_{\P^n}(d_i)|_{X_\ud}$.

\

Note that in these statements the varieties are assumed to be {\em
smooth}. For singular varieties the situation is more complicated
and odd homologies can be non-zero.

\subsection{Algebraic cell structure}
Let $X$ be a compact topological space and let $X=\coprod\sigma_\alpha$ be
a cell decomposition (each cell $\sigma_\alpha$ is
homeomorphic to some $\R^k$, the smaller cells are glued to the
bigger ones by the boundary maps). \bex 1. $S^n=\R^n\cup R^0$.
\\2. For the complex projective space $\Pr^n$ we have  the following cell
decomposition: $\Pr^n=\C^n\cup\C^{n-1}\cup\dots\cup\C^0$
(see e.g. \cite[1.5]{griffiths:94}). \eex If
$X$ is a complex algebraic variety (possibly singular) then it is
natural to ask for an {\em algebraic} cell structure, so that each
cell, $\sigma_\alpha$, is itself a subvariety of $X$, isomorphic
to $\C^n$. For example, the cell structure of $\Pr^n$ is of this
type. Though such an algebraic cell decomposition is not always
possible, it often occurs in our context.

Suppose a variety admits an algebraic cell decomposition,
$X=\coprod (k_i\C^i)$, here $k_i$ is the number of cells of the
given dimension. This fixes the homology. Indeed, all the cells
are of even (real) dimensions. Thus all the boundary maps are
zero.  So, $H_{2i}(X,\Z)=\Z^{k_i}$ for $0\le 2i\le \dim_{\R}X$
and $H_{m}(X,\Z)=0$ for other $m$.

\subsection{Exact sequence of a pair}\label{Sec.Exact.Sequence.of.Pair}
For a pair of topological spaces, $A\sset X$, there are exact
sequences: \beq \cdots\to H_i(A,\Z)\to H_i(X,\Z)\to H_i(X,A,\Z)\to
H_{i-1}(A,\Z)\cdots,\quad \cdots\to H^i(X,\Z)\to H^i(A,\Z)\to
H^{i+1}(X,A,\Z)\to \cdots,\quad \eeq

\subsection{(Co)homology of non-compact spaces}
\bprop\cite[pg.157]{ffuks}\label{Thm.Cohom.Non.Compact.Spaces}
Let $X$ be a compact topological space, let $A\sset X$ be its
closed subspace, such that $X\smin A$ is smooth, orientable,
without boundary, of (real) dimension $n$. Then $H_i(X,A,\Z)\isom
H^{n-i}(X\smin A,\Z)$ \eprop (This follows from the definition of
Borel-Moore homology, via compactification: $H_i^{BM}(X\smin
A,\Z)=H_i(X,A,\Z)$ and the Poincare duality for non-compact
manifolds: $H^{BM}_i(M,\Z)\isom H^{n-i}(M,\Z)$.)

\section{}
\label{app2}

Let us recall Proposition 5.16 from \cite{mill:04}:
\par\medskip
\par\medskip
\noindent {\bf Proposition 5.16} {\it Let $2k+3 \ge 9$ then
$H^2_{(2k+4)}(\mathrm{m}_{0,3}(2k+3))=0$ and therefore
$\mathrm{m}_{0,3}(2k+3)$ has  no filiform central extensions.}
\par\medskip
\par\medskip
In the above proposition $\mathrm{m}_{0,3}(2k+3)$ denotes a
$(2k+3)$-dimensional $\mathbb{N}$-graded filiform Lie algebra with
the basis: $e_1, e_2,\ldots, e_{2k+3}$ and multiplication table:
\[ [e_1, e_i] =e_{1+i}, \, i=2,\ldots, 2k+2;\]
\[ [e_l, e_{2k+1-l}] =(-1)^{l+1}e_{2k+1}, \, l=2,\ldots, k;\]
\[ [e_j, e_{2k+2-j}] = (-1)^{j+1} (k-j+1) e_{2k+2},\, j=2,\ldots,
k;\]
\[ [e_m, e_{2k+3-m}] = (-1)^m ( (m-2)k - \frac{(m-2)(m-1)}{2} )
e_{2k+3}, \, m=3,\ldots, k+1.\]

In fact, this proposition holds true for any $k>3$. If $k=3$, then
$\mathrm{m}_{0,3}(9)$ does have a filiform central extension which
is $\mathrm{m}_{0,4}(10)$ with the basis: $e_1, e_2,\ldots, e_9,
e_{10}$ and multiplication table:
\begin{align*}
& [e_1, e_i] = e_{1+i}, \, i=2,\ldots 9,\\
& [e_2, e_5] = -e_7,\, [e_3, e_4] = e_7,\\
& [e_2, e_6] = -2e_8,\, [e_3, e_5] = e_8,\\
& [e_3, e_6] = -2e_9,\, [e_4, e_5] = 3e_9,\\
& [e_4, e_6] = 3e_{10}, \, [e_3, e_7] = -5 e_{10},\,
[e_2,e_8] = 5e_{10}.\\
\end{align*}

In its turn, $\mathrm{m}_{0,4}(10)$ has also filiform central
extension, which is  $\mathrm{m}_{0,5}(11)=\text{span}\{e_1,e_2,
\ldots, e_{10}, e_{11}\}$ with the following multiplication table:

\begin{align*}
& [e_1, e_i] = e_{1+i}, \, i=2,\ldots 10,\\
& [e_2, e_5] = -e_7,\, [e_3, e_4] = e_7,\\
& [e_2, e_6] = -2e_8,\, [e_3, e_5] = e_8,\\
& [e_3, e_6] = -2e_9,\, [e_4, e_5] = 3e_9,\\
& [e_4, e_6] = 3e_{10}, \, [e_3, e_7] = -5 e_{10},\,
[e_2,e_8] = 5e_{10},\\
& [e_3, e_8] = \frac{5}{2} e_{11}, \, [e_2, e_9] = \frac{5}{2}
e_{11},\, [e_4,e_7] = -\frac{15}{2}e_{11},\, [e_5, e_6]=\frac{21}{2} e_{11}.\\
\end{align*}

The latter algebra has no filiform central extensions.

\end{document}